\documentclass[a4paper,10pt]{article}
\usepackage[utf8]{inputenc}

\usepackage{hyperref}
\usepackage{ifpdf}

\usepackage[all]{xy}

\usepackage{amsmath}
\usepackage{amsthm}
\usepackage{amssymb}

\usepackage{relsize}

\usepackage[T1]{fontenc}

\usepackage{mathrsfs}
\usepackage{eucal}

\usepackage[ugly]{units}

\usepackage[affil-it]{authblk} 

\usepackage{wrapfig}

\ifpdf
	\usepackage[pdftex]{graphicx}
	\usepackage{epstopdf}
\else
	\usepackage{graphicx}
\fi

\newtheorem{theorem}{Theorem}[section]
\newtheorem{lemma}[theorem]{Lemma}
\newtheorem{proposition}[theorem]{Proposition}

\theoremstyle{definition}
\newtheorem{definition}{Definition}[section]
\newtheorem{remark}{Remark}[section]

\newtheorem{algorithm}{Algorithm}[section]

\newcommand{\set}[1]{\mathsf{#1}}
\newcommand{\ud}{\, \mathrm{d} }

\newcommand{\pd}{ \partial}
\newcommand{\e}{\mathrm{e}}

\newcommand{\realset}{\mathbb{R}}

\newcommand{\inner}[1]{\langle\!\langle #1 \rangle\!\rangle}
\newcommand{\pair}[1]{\langle #1 \rangle}
\providecommand{\norm}[1]{\lVert#1\rVert}
\providecommand{\abs}[1]{\lvert#1\rvert}
\newcommand{\Id}{\mathrm{Id}}

\newcommand{\vect}[1]{\boldsymbol{#1}}

\newcommand{\vq}{\vect{q}}

\newcommand{\vp}{\vect{p}}

\newcommand{\T}{T}
\newcommand{\manifold}[1]{\mathscr{#1}} 
\newcommand{\Q}{\manifold{Q}}
\newcommand{\Smanifold}{\manifold{S}}
\newcommand{\TQ}{\T\!\Q}
\newcommand{\coTQ}{\T^{*}\!\Q}

\newcommand{\PP}{\manifold{P}}
\renewcommand{\P}{\PP}

\newcommand{\smoothfun}{\mathcal{F}}
\newcommand{\Fcal}{\smoothfun}
\newcommand{\Xcal}{\mathfrak{X}}

\newcommand{\Diff}{\mathrm{Diff}}

\newcommand{\DiffSp}{\Diff_\mathrm{Sp}}
\newcommand{\DiffHam}{\Diff_\mathrm{Ham}}

\newcommand{\XcalSp}{\Xcal_\mathrm{Sp}}
\newcommand{\XcalHam}{\Xcal_\mathrm{Ham}}

\newcommand{\pois}[1]{\{#1\}}

\newcommand{\LieD}[1]{\pounds_{#1}}

\newcommand{\interior}[1]{\mathrm{i}_{#1}}

\newcommand{\trans}{\top}

\newcommand{\Ad}{\mathrm{Ad}}

\newcommand{\G}{\manifold G}
\newcommand{\Gsub}{\manifold H}

\newcommand{\g}{\mathfrak{g}}
\newcommand{\gsub}{\mathfrak{h}}
\newcommand{\cog}{\mathfrak{g}^*}

\newcommand{\eps}{\varepsilon}

\newcommand{\Mcal}{\mathcal{M}}
\newcommand{\Dcal}{\mathcal{D}}

\newcommand{\DiffG}{\Diff_{\G}}
\newcommand{\XcalG}{\Xcal_{\G}}
\newcommand{\DiffJ}{\Diff_{\set{J}}}
\newcommand{\XcalJ}{\Xcal_{\set{J}}}
\newcommand{\DiffGJ}{\Diff_{\G,\set{J}}}
\newcommand{\XcalGJ}{\Xcal_{\G,\set{J}}}
\newcommand{\mybreve}[1]{{#1}^\Delta}
\newcommand{\PhiThreeS}{\Phi_{\textrm{\footnotesize 3S}}}
\newcommand{\PhiTwoS}{\Phi_{\textrm{\footnotesize 2S}}}
\newcommand{\PhiRKS}{\Phi_{\textrm{\footnotesize RKS}}}
\newcommand{\PhiSV}{\Phi_{\textrm{\footnotesize SV}}}

\title{Geometric Integration of Hamiltonian Systems \\ Perturbed by Rayleigh Damping}
\author{Klas MODIN$^{a}$}
\author{Gustaf SÖDERLIND$^{b}$}
\affil{
	$^a$School of Engineering and Advanced Technology, Massey University \\
	Private Bag 11~222, Palmerston North 4442, New Zealand \\
	E--mail: {\selectfont\ttfamily\itshape k.e.f.modin@massey.ac.nz}
}
\affil{
	$^b$Centre for Mathematical Sciences, Lund University \\
	Box 118, SE--221 00 Lund, Sweden \\
	E--mail: {\selectfont\ttfamily\itshape gustaf.soderlind@na.lu.se}
}

\date{\today}

\newcommand{\keywords}[1]{\par\smallskip {\bf Key-words:} #1}
\newcommand{\subclass}[1]{\par\smallskip {\bf 2000 MSC:} #1}
\begin{document}

\maketitle

\begin{abstract}
Explicit and semi-explicit geometric integration schemes for dissipative
perturbations of Hamiltonian systems 
are analyzed.
The dissipation is characterized by a small parameter $\varepsilon$, 
and the schemes under study preserve the symplectic structure in the case $\varepsilon=0$.
In the case $0<\varepsilon\ll 1$ the energy dissipation rate is shown
to be asymptotically correct by backward error analysis.
Theoretical results on monotone decrease of the modified
Hamiltonian function for small enough step sizes are given.
Further, an analysis proving near conservation of
relative equilibria for small enough step sizes is conducted.

Numerical examples, verifying the analyses, are given
for a planar pendulum and an elastic 3--D~pendulum.
The results
are superior in comparison with a conventional explicit Runge-Kutta method
of the same order.

\subclass{65P10
	\and 37M15
	\and 70F40
}

	%
\end{abstract}


\newpage
\tableofcontents


\section{Introduction} 
\label{sec:Introduction}

In this paper we analyze geometric numerical integration algorithms for
weakly dissipative perturbations of Hamiltonian systems. 
The dissipation rate is governed by a small parameter 
$0\leq\varepsilon\ll 1$, as in the simple example
\begin{equation}\label{eq:simple_example_equation}
	\begin{split}
		\dot \vq & = M^{-1} \vp \\
		\dot \vp & = -V'(\vq) - \eps D \dot\vq
	\end{split}	
\end{equation}
where $\vq,\vp\in\realset^n$, $M$~is a symmetric positive definite matrix, and
$D$ is a symmetric positive semi-definite matrix.
At $\varepsilon =0$ the system is Hamiltonian with
$H(\vq,\vp)=\frac{1}{2}\vp^\trans M^{-1} \vp+V(\vq)$,
and for $\varepsilon>0$ it is dissipative.
This type of dissipation is the simplest example of \emph{Rayleigh damping}
(the general case is presented in Section~\ref{sec:problem_description} below).

There are many applications. 
For example, in rolling bearing design, 
the objective is to minimize friction losses. 
The dynamics of such mechanical systems is often well described by this model. 
In numerical simulations, it is then of importance that energy losses are correct. 
This means that a geometric integrator (nearly) conserving energy at $\varepsilon = 0$, 
and giving the (nearly) correct dissipation rate as a function of $\varepsilon>0$, is desirable. 
Moreover, if possible, explicit methods are preferred. However, well-known reversible explicit methods,
such as the two-step explicit midpoint rule, 
typically lose their stability in the presence of dissipation. 
Furthermore, explicit symplectic methods for separable Hamiltonian systems
(e.g.~Lobatto~IIIA--IIIB)
become implicit if dissipation is added.
This motivates a special study and analysis of explicit geometric integrators
for systems with weak Rayleigh damping.

The paper will address this question by examining splitting methods that take advantage 
of the linearity of the dissipative perturbation.
Three integration schemes are investigated.
All of these schemes are explicit in the sense that
no numerical root solving is required.
The first two of the methods are
exponential integrators (meaning they
require the solution of a linear ODE).
The third method is either fully explicit,
or linearly implicit (meaning it requires the solution of
a linear equation). In our numerical examples,
the three methods under study turn out to yield
practically identical results. Although the presented methods
are all second order accurate, it is an easy
extension to construct higher order methods 
by using higher order compositions.
By backward error analysis we show that the
methods have asymptotically correct
dissipation behavior as the dissipation parameter~$\eps\to 0^+$.
Moreover, we analyze the decay of the modified
Hamiltonian function, and give theoretical results on monotonicity
of its evolution for small enough step sizes.
We also analyze conservation of momentum, and near conservation
of relative equilibria, in the case when both the conservative
and the dissipative parts are invariant under a symmetry.

The examined methods are invariant under
choice of coordinates in the configuration space.
Therefore, they are well defined
on any smooth configuration manifold without reference
to a specific choice of coordinates.
Throughout this paper we utilize the differential
geometric framework of Riemannian geometry to describe
the systems of interest.
Although an important consequence, the main reason for doing so 
is~\emph{not} to include
systems on non-linear spaces (constrained systems). Rather,
the presented framework and corresponding analysis
become more transparent in a differential geometric setting. 
Also, the fact that the result
are derived for abstract manifolds is essential for some
of the analysis carried out in the paper. 
(For example, independence of configuration coordinates
is a key ingredient in our proof of preservation of
symmetry invariance.) Moreover, well known results in
differential geometry, in particular on structural stability, symmetry groups
and momentum maps, and relative equilibria, 
become more easily available in this setting.
For the convenience of the reader, we sometimes carry
out invariant calculations even in cases when
they are evident from a differential geometric point of view.

%
%
%
%

The idea of applying geometric integration schemes to
weakly dissipative problems is certainly not new.
For instance, it is well known that
symplectic methods 
applied to perturbed integrable systems
(e.g.\ the Van der Pols equation) are superior to
standard methods, in the sense that weakly attractive invariant
tori are much better preserved; see~\cite[Chap.~XII]{HaLuWa2006}.
(Although the concepts of relative equilibria and invariant tori are closely related,
the latter can be seen as a special case of the former,
the settings and analyses in this paper and in~\cite[Chap.~XII]{HaLuWa2006} are not
overlapping. A comparison between the two settings is given in
Section~\ref{sub:relation_to_symplectic_integration_of_problems_with_attracting_invariant_tori}
below.)
%
Furthermore, in the paper~\cite{KaMaOrWe2000} it is observed numerically
that variational methods, in particular
a semi-explicit version of the Newmark method,
applied to weakly dissipative
systems give much more accurate energy dissipation rates than standard methods.
In the simplified case of \emph{conformal Hamiltonian systems}, 
corresponding to $D=\Id$ in equation~\eqref{eq:simple_example_equation}
above, it is possible to construct integrators that exactly preserve the 
corresponding geometric structure;
see~\cite{McPe2001,McQu2006}.
Concerning preservation of relative equilibria and numerical integration,
there has been extensive work within the framework
of energy--momentum methods; see~\cite{GoSi1996,Go1996,ArRo2001a,ArRo2001b}.
Under suitable conditions, 
energy--momentum methods exactly preserve relative equilibria.
However, energy momentum schemes are fully implicit.
Also related to our work is the recent paper~\cite{LuWaBr2010}, 
which uses a splitting technique,
similar to the one used in this paper, for the
integration of post-Newtonian equations,
possibly with weak dissipation due to
relativistic effects.

The paper is organized as follows. A problem description is given
in Section~\ref{sec:problem_description}. Numerical integration
schemes are presented in Section~\ref{sec:numerical_integration_schemes}.
A backward error analysis of the methods is 
given in Section~\ref{sec:backward_error_analysis}.
The result of the analysis is verified
by a numerical example (planar pendulum) in
Section~\ref{sub:example_damped_pendulum}.
In Section~\ref{sec:conservation_of_momentum} we show that
the methods preserve symmetry invariance, and conserve corresponding
momentum maps. In conjunction,
Section~\ref{sec:near_conservation_of_relative_equilibria}
contains a study of
conditions for near conservation of relative equilibria,
verified by a numerical example (elastic pendulum) in Section~\ref{sub:elastic_pendulum}.
Finally, conclusions are given in Section~\ref{sec:conclusions}.

Throughout the paper
we assume that the reader is familiar with basic
differential geometry. In particular, the following
geometric concepts are used without reference: Riemannian manifolds,
tangent and co-tangent bundles,
tensor fields, Lie derivative, symplectic maps, Hamiltonian vector fields, 
flow of a vector field, exponential map.
For the convenience of
the reader, other more advanced, but still fundamental, geometric concepts
are introduced and defined as they are used.
The introductory textbook on geometric mechanics 
by Marsden and Ratiu~\cite{MaRa1999}
is more than enough to cover all the differential geometric concepts in this paper.

We adopt the following notation. 
$\Q$ denotes a manifold, $\TQ$
and $\coTQ$ its corresponding tangent and co-tangent bundle.
Throughout the paper $q$ denotes an element in $\Q$.
Furthermore, $p$ denotes an element in
the co-tangent space $\T^*_q\Q$.
Thus, $p$ is always associated with~$q$.
The pairing between an element $p\in\T^*_q\Q$ and $\dot q\in\T_q\Q$
is denoted~$\pair{p,\dot q}$.
We use $z$ for a point in $\coTQ$ corresponding to~$(q,p)$.
The derivative of a function $\phi$ on $\Q$
at the point $q$ (also called the tangent map)
is denoted~$\T_q\phi$.
The set of smooth vector field on a manifold~$\P$ is denoted~$\Xcal(\P)$,
and the set of diffeomorphisms of~$\P$ is denoted~$\Diff(\P)$.
The flow of a vector field~$X\in\Xcal(\P)$ is~$\exp(t X)$, where~$t$ 
is the time parameter and~$\exp:\Xcal(\P)\to\Diff(\P)$ is the exponential map. 
The Lie derivate along~$X$ is denoted~$\LieD{X}$.
The space of smooth real-valued functions
on~$\P$ is denoted~$\Fcal(\P)$.
If $H\in\Fcal(\coTQ)$, then its corresponding
canonically Hamiltonian vector field is denoted~$X_H$.
The canonical Poisson bracket between~$H,G\in\Fcal(\coTQ)$
is denoted~$\pois{H,G}$.
The space of type~$(r,s)$ tensor fields over~$\Q$ 
is denoted~$\mathcal{T}^r_s(\Q)$.

%
%




\section{Problem Description} 
\label{sec:problem_description}

Let $(\Q,\mathsf{g})$ be a Riemannian manifold of dimension~$n$,
where $\mathsf{g}\in\mathcal{T}_2^0(\Q)$ is a metric tensor, i.e., symmetric and positive definite.
In our context~$\Q$
is the configuration space of a conservative
mechanical system with Lagrangian function of standard 
form~$L(\dot q,q) = \frac{1}{2}\mathsf{g}_q(\dot q,\dot q) - V(q)$.
Thus, the metric tensor determines the kinetic energy
of the system.

We are interested in dissipative
perturbations of such systems. To this end, let
$\mathsf{d}\in\mathcal{T}_2^0(\Q)$ be a symmetric positive semi-definite
tensor field called \emph{dissipation tensor}.
We shall assume that
the rank of~$\mathsf{d}$ is constant, i.e.,
that it is a regular tensor.
Further, the \emph{Rayleigh dissipation function} is defined by
$R(q,\dot q) = \frac{1}{2}\mathsf{d}_{q}(\dot q,\dot q)$.
We consider the following dissipative perturbation
of the Euler--Lagrange equation
\begin{equation}\label{eq:dissipative_Euler_Lagrange}
	\frac{\ud}{\ud t}\frac{\pd L}{\pd \dot q} - \frac{\pd L}{\pd q} =
	-\eps \frac{\pd R}{\pd \dot q} ,
\end{equation}
where 
$\eps\geq 0$ is the perturbation parameter.



From a structural, as well as a numerical, point of view it is
often preferable to transform the Euler--Lagrange equation into
a first order Hamiltonian system on 
the co-tangent bundle phase space~$\coTQ$. 
The metric tensor induces an isomorphism $\Mcal(q):\T_q\Q \to \T^*_q\Q$ 
by the \emph{flat operator} 
$\dot q\mapsto \dot q^\flat := \mathsf{g}_q(\dot q,\cdot)$.
Its inverse defines the \emph{sharp operator} 
$p\mapsto {p}^\sharp:=\Mcal(q)^{-1}p$. 
$\Mcal$ is sometimes called the \emph{inertia operator}.
Since the kinetic energy descends from the metric tensor,
the Legendre transformation is given by the flat operator,
so the momentum variable is $p :=\pd L/\pd\dot q=\dot q^\flat$.
%
%
The Hamiltonian function, describing the total energy of the system, is given by
the sum of kinetic and potential energies
\begin{equation}\label{eq:Hamiltonian_function}
	H(q,p) = \underbrace{\frac{1}{2}\mathsf{g}_q(p^\sharp,p^\sharp)}_{T(q,p)}+V(q).
\end{equation}

Substituting variables,
the perturbed Euler--Lagrange equation~\eqref{eq:dissipative_Euler_Lagrange} transforms
into a perturbed Hamiltonian system
\begin{equation}\label{eq:perturbed_Hamiltonian_system}
	\begin{split}
		\dot q &= p^\sharp \\
		\dot p &= -\frac{\pd H}{\pd q} - \eps\,\mathsf{d}_q(p^\sharp,\cdot) \, ,
	\end{split}
\end{equation}
where we have used that $\pd H/\pd p = p^\sharp$, which follows since
$\mathsf{g}_q(p^\sharp,p^\sharp) = \pair{p,p^\sharp}$.



%
%

Since the Hamiltonian part of equation~\eqref{eq:perturbed_Hamiltonian_system}
is perturbed by a non-Hamiltonian perturbation, total energy is not conserved.
Indeed, the time evolution of the total energy along solution curves is given by
\begin{equation}\label{eq:exact_energy_dissipation_rate}
	\frac{\ud}{\ud t}H(q,p) = \frac{\pd H}{\pd q}\dot q + \frac{\pd H}{\pd p}\dot p
	= -\eps\, \mathsf{d}_q(p^\sharp,p^\sharp) = -\eps\, \mathsf{d}_q(\dot q,\dot q) \leq 0 .
\end{equation}
Thus, the rôle of the dissipation tensor is to specify the rate of
energy decay at each point in phase space.

The form of system~\eqref{eq:perturbed_Hamiltonian_system} is
invariant under point transformations, as we will now derive.
Indeed, let~$\Q$ and~$\Smanifold$ be diffeomorphic manifolds and $\phi:\Smanifold\to\Q$
a diffeomorphism. Consider now the change of variables $(s,r)  = \T^*\phi (q,p)$,
where $\T^*\phi:\coTQ\to\T^*\Smanifold$ is the \emph{co-tangent lift} of~$\phi$
(cf.~Marsden and Ratiu~\cite[Sect.~6.3]{MaRa1999}).
(In coordinates, $(\vect s,\vect r)\leftrightarrow (\vect q,\vect p)$ is defined by the
generating function of second kind $G_2(\vect{q},\vect{r})=\phi(\vect{q})\cdot\vect{r}$.)
We know already from standard Hamiltonian theory that the governing equation
for the Hamiltonian part of the vector field in the variables $(s,r)$
transforms into a new canonical Hamiltonian system for a new
Hamiltonian function $K$ on $\T^*\Smanifold$ given by
\[
	\begin{split}
		K(s,r) &= H\big(\T^*\phi^{-1}(s,r)\big) 
		\\
		& = 
		\pair{(\T_{s}\phi)^{-1*}r,\Mcal(\phi(s))^{-1}(\T_{s}\phi)^{-1*}r}
		+ V(\phi(s)) \\
		&= \pair{r,\underbrace{(\T_{s}\phi)^{-1}\Mcal(\phi(s))^{-1}(\T_{s}\phi)^{-1*}r}_{r^\sharp}}
		+ V(\phi(s)) \\
		&= \pair{r,r^\sharp} + (\phi^* V)(s) \\
		& = (\phi^*\mathsf{g})_s(r^\sharp,r^\sharp) + (\phi^* V)(s)
	\end{split}
\]
where $(\phi^*\mathsf{g})_{s}(u,v) = 
\mathsf{g}_{\phi(s)}(\T_s\phi \, u,\T_s\phi \, v)$ and $(\phi^* V)(s) = V(\phi(s))$
is the \emph{pull-back} of $\mathsf{g}$ and $V$.
Thus, in the new variables $(s,r)$ the Hamiltonian part of
system~\eqref{eq:perturbed_Hamiltonian_system} takes again the same form 
but with the pull-backed quantities $\phi^*\mathsf{g}$ and $\phi^* V$.
Next, the vector field over~$\coTQ$ corresponding to the dissipation
is given by $\eps Y(q,p) = (0,-\eps\mathsf{d}_q(p^\sharp,\cdot))=:(0,-\eps\Dcal(q)p^\sharp)$,
where $\Dcal(q):\T_q\Q\to\T^*_q\Q$. 
The corresponding vector field in the $(s,r)$ variables is
the \emph{push-forward} of~$Y$ by~$\T^*\phi$, which is given by 
$(\T^*\phi)_* Y := \T(\T^*\phi)\circ X \circ \T^*\phi^{-1}$.
In matrix notation
\begin{equation} \label{eq:pushforward_calculation}
	\begin{split}
		(\T^*\phi)_* Y(s,r) &= 
		\begin{bmatrix}
			\T_s \phi^{-1} & 0 \\
			\frac{\pd}{\pd s}(\T_s\phi)^*r & (\T_s\phi)^*
		\end{bmatrix}
		\begin{bmatrix}
			0 \\
			\Dcal(\phi(s))(\T_{s}\phi)^{-1*}r
		\end{bmatrix}
		\\ & = 
		\begin{bmatrix}
			0 \\
			(\T_s\phi)^*\Dcal(\phi(s))\Mcal(\phi(s))^{-1}(\T_{s}\phi)^{-1*}r
		\end{bmatrix}
		\\ & = 
		\begin{bmatrix}
			0 \\
			\underbrace{(\T_s\phi)^*\Dcal(\phi(s)) \T_s\phi}_{(\phi^*\mathsf{d})_s(\cdot,\cdot)}
				\underbrace{(\T_s\phi)^{-1}\Mcal(\phi(s))^{-1}(\T_{s}\phi)^{-1*}r}_{r^\sharp}
		\end{bmatrix}
		\\ & =
		\big(0,(\phi^*\mathsf{d})_s(r^\sharp,\cdot)\big)
	\end{split}
\end{equation}
In summary, this means that
system~\eqref{eq:perturbed_Hamiltonian_system} is determined by
the three geometric objects (tensors) $\mathsf{g}$,~$\mathsf{d}$~and~$V$ 
in such a way that
%
a change of
variables $(s,r)  = \T^*\phi (q,p)$ takes system~\eqref{eq:perturbed_Hamiltonian_system}
into a new system of the same form defined by the pulled back 
triple $\phi^*\mathsf{g}$,~$\phi^*\mathsf{d}$~and~$\phi^*V$. 
%
In other words, the following diagram commutes:
\begin{equation}\label{eq:commutative_coordinate_independence}
	\xymatrix@=10ex{\mathsf{g},\mathsf{d},V 
	\ar[r]^{\text{\normalsize $\phi^*$}}\ar[d]_{\text{\rm\normalsize flow}} &
	\ar[d]^{\text{\rm\normalsize flow}} \mathsf{h},\mathsf{e},W \\
	q(t),p(t) \ar[r]^{\text{\normalsize $\T^*\phi$}} & s(t),r(t)}
\end{equation}

Thus, the form of equation~\eqref{eq:perturbed_Hamiltonian_system} is independent
of the choice of coordinates on $\Q$. The form is not, however, invariant
with respect to any symplectic coordinate transformation on the full phase space $\coTQ$:
only 
point transformations are allowed.
As we will see, independence of coordinates on $\Q$
is preserved by the numerical schemes suggested in Section~\ref{sec:numerical_integration_schemes}.

%
%
%
%


\section{Numerical Integration Schemes} 
\label{sec:numerical_integration_schemes}

In this section we present the explicit and semi-explicit time-stepping schemes
for the system~\eqref{eq:perturbed_Hamiltonian_system} above to be analyzed in the paper.
%
To begin with, we rewrite equation~\eqref{eq:perturbed_Hamiltonian_system}
in the more compact form
\begin{equation}\label{eq:perturbed_Hamiltonian_system_condense}
	\dot z = X_T(z) + X_V(z) + \eps Y(z), \qquad z = (q,p)\in\coTQ ,
\end{equation}
where $X_T,X_V$ are the Hamiltonian vector fields corresponding
to $T$ and $V$ respectively, and $\eps Y$ is the non-Hamiltonian
Rayleigh dissipation part.

Since the potential energy function
$V$ is independent of~$p$, the flow $\exp(tX_V)$ of $X_V$ is explicitly
integrable for any choice of coordinates on~$\Q$ simply by applying
the forward Euler method. Throughout the rest of this paper we make
one assumption which is crucial for the implementation
of the suggested methods: that we know coordinates
on~$\Q$ in which the flow $\exp(t X_T)$ is explicitly computable.
For particle systems, this is typically accomplished by choosing
Cartesian coordinates in which the inertia operator~$\Mcal$
is independent of~$q$. In this case, $\exp(t X_T)$ is again
computable using the forward Euler method.

We study the following schemes:

\begin{algorithm}[Three-term splitting method] \label{alg:3S}
	Let $h>0$ be the step size and define $z_k \mapsto z_{k+1}$
	by the map
	\[
		\PhiThreeS^h = \exp(\frac{h}{2}\eps Y)\circ\exp(\frac{h}{2} X_T)\circ
		\exp(h X_V)\circ\exp(\frac{h}{2}X_T)\circ\exp(\frac{h}{2}\eps Y) .
	\]
\end{algorithm}

\begin{algorithm}[Two-term splitting method] \label{alg:2S}
	Let $h>0$ be the step size and define $z_k \mapsto z_{k+1}$
	by the map
	\[
		\PhiTwoS^h = \exp(\frac{h}{2} X_T)\circ
		\exp\big(h (X_V+\eps Y)\big)\circ\exp(\frac{h}{2}X_T) .
	\]
\end{algorithm}

\begin{algorithm}[Runge-Kutta splitting method] \label{alg:RKS}
	Let $h>0$ be the step size and let $\Psi_{X}^h$
	be a Runge-Kutta method for $X\in\Xcal(\coTQ)$.
	Define $z_k \mapsto z_{k+1}$ by the map
	\[
		\PhiRKS^h = \exp(\frac{h}{2} X_T)\circ
		\Psi_{X_V+\eps Y}^h\circ\exp(\frac{h}{2}X_T) .
	\]
\end{algorithm}

All of these methods reduce to the classical Störmer-Verlet method~$\PhiSV^h$
(also known as the leap-frog method or the Verlet scheme)
in the case~$\eps=0$.
(To be more precise, they reduce to the dual version, 
or $B$--version, of the classical Störmer-Verlet method; 
see~\cite{HaLuWa2003} for various interpretations of this method.)
Furthermore, Algorithm~\ref{alg:3S} and Algorithm~\ref{alg:2S}
are second order accurate, as is Algorithm~\ref{alg:RKS} if the Runge-Kutta method used
is at least of order 2. It is straightforward to extend these methods
to higher order, by using higher order compositions
(see~\cite{McQu2002} for a review of splitting methods).

It should be noted that the vector field $X_V+\eps Y$ 
is linear in~$p$ and thus explicitly integrable
using the exponential Euler method (cf.~\cite{HoOs2010}).
Hence, $\PhiTwoS^h$ is explicitly computable.
Likewise,~$Y$ is explicitly integrable using the exponential
Euler method, so $\PhiThreeS^h$
is also explicitly computable. If the Runge-Kutta method in Algorithm~\ref{alg:RKS}
is explicit, then $\PhiRKS^h$ is fully
explicit. If the Runge-Kutta method used is implicit,
then each step of $\PhiRKS^h$ requires the solution
of a linear system.

Runge-Kutta methods are well defined on linear spaces. Thus,
one might think that Algorithm~\ref{alg:RKS} is defined only
in the case when $\coTQ\simeq \realset^{2n}$. This is however
not the case,
since the dynamics of $X_V+\eps Y$ is trivial in the configuration variable~$q$,
so each ``Runge-Kutta step''~$\Psi^h_{X_V+\eps Y}$ takes place
entirely on a co-tangent space $\T^*_q\Q$, which is a linear space.
Furthermore, since Runge-Kutta methods are invariant with respect
to linear coordinate changes, and since any smooth change of configuration
coordinates corresponds to a linear change of coordinates on each co-tangent space,
Algorithm~\ref{alg:RKS} is defined intrinsically (i.e., without reference to a specific choice
of configuration coordinates).

\begin{remark}
	Notice that system~\eqref{eq:perturbed_Hamiltonian_system_condense}
	is of the form
	\[
		\begin{split}
			\dot q &= f(q,p) \\
			\dot p &= g(q,p) 
		\end{split}
	\]
	so a partitioned Runge-Kutta method may be used
	to integrate it. The classical Störmer-Verlet method
	can be extended to the partitioned Runge-Kutta method
	given by the second order Lobatto~IIIA--IIIB pair;
	see~\cite[Sect.~II.2]{HaLuWa2006}.
	From our point of view, assuming that we are using a coordinate
	system in which the kinetic energy is independent of~$q$, this extension
	of the classical Störmer-Verlet method is conjugate to
	Algorithm~\ref{alg:RKS} with the implicit midpoint rule
	as choice of Runge-Kutta method.
	Indeed, if $\Phi^h_\mathrm{M}$, $\Phi^h_\mathrm{I}$ and $\Phi^h_\mathrm{E}$
	denotes, respectively, the implicit midpoint rule, the implicit Euler method and
	the explicit Euler method, it holds that
	$\Phi^h_\mathrm{M}=\Phi^{h/2}_\mathrm{E}\circ\Phi^{h/2}_\mathrm{I}$.
	Thus, Algorithm~\ref{alg:RKS} with the implicit midpoint rule as the choice of
	Runge-Kutta method is 
	$\exp(h X_T/2)\circ\Phi^{h/2}_\mathrm{E}\circ\Phi^{h/2}_\mathrm{I}\circ\exp(h X_T/2)$.
	This method is conjugate to $\Phi^{h/2}_\mathrm{I}\circ\exp(h X_T)\circ\Phi^{h/2}_\mathrm{E}$,
	which is exactly the $A$--version of the partitioned Runge-Kutta
	extension of the classical Störmer-Verlet method.
	%
	%
	%
	%
	%
\end{remark}

As mentioned above, the suggested integrators are invariant with respect to
choice of coordinates on the configuration space. From a differential geometric
point of view, this result is obvious, since all the objects and
operations used in their definition are intrinsic (configuration coordinate independent).
However, for the convenience of the reader, we carry out
a proof with complete calculations.

\begin{proposition}\label{pro:coordinate_invariance}
	Algorithms~\ref{alg:3S}--\ref{alg:RKS} are invariant with respect to
	choice of coordinates in $\Q$. That is, if $\Smanifold\simeq\Q$ 
	and $\phi:\Smanifold\to\Q$ is a diffeomorphism,
	then the following diagram commute:
	\[
	\xymatrix@=10ex{\mathsf{g},\mathsf{d},V 
	\ar[r]^{\text{\normalsize $\phi^*$}}\ar[d]_{\text{\rm\normalsize method}} &
	\ar[d]^{\text{\rm\normalsize method}} \mathsf{h},\mathsf{e},W \\
	\{(q_k,p_k)\}_k \ar[r]^{\text{\normalsize $\T^*\phi$}} & \{(s_k,r_k)\}_k}
	\]
\end{proposition}

\begin{proof}
	For simplicity we denote the map $\T^*\phi:\coTQ\to\T^*\Smanifold$ by~$\chi$.
	Further, we denote Algorithms~\ref{alg:3S}--\ref{alg:RKS} on $\coTQ$ for the triple
	$(\mathsf{g},\mathsf{d},V)$ by 
	$\Phi^h_\mathrm{3S}$, 
	$\Phi^h_\mathrm{2S}$ and
	$\Phi^h_\mathrm{RKS}$. 
	The corresponding algorithms on $\T^*\Smanifold$
	for the pulled back triple 
	$(\mathsf{h},\mathsf{e},W):=(\phi^*\mathsf{g},\phi^*\mathsf{d},\phi^*V)$
	are denoted 
	$\Upsilon^h_\mathrm{3S}$, 
	$\Upsilon^h_\mathrm{2S}$ and
	$\Upsilon^h_\mathrm{RKS}$. 
	We need to show that 
	$\Phi^h_\mathrm{3S}=\chi\circ\Upsilon^h_\mathrm{3S}\circ\chi^{-1}$, 
	$\Phi^h_\mathrm{2S}=\chi\circ\Upsilon^h_\mathrm{2S}\circ\chi^{-1}$ and
	$\Phi^h_\mathrm{RKS}=\chi\circ\Upsilon^h_\mathrm{RKS}\circ\chi^{-1}$.
	
	First, notice that the vector fields $X_T$, $\eps Y$ and $X_V$ on $\coTQ$ 
	are defined respectively by the triples 
	$(\mathsf{g},0,0)$, $(0,\mathsf{d},0)$ and $(0,0,V)$.
	Thus, from the commutative diagram~\eqref{eq:commutative_coordinate_independence}
	it follows that $\exp(t X_U) = \chi\circ\exp(t X_T)\circ\chi^{-1}$,
	$\exp(t Y) = \chi\circ\exp(t Z)\circ\chi^{-1}$ and 
	$\exp(t X_V) = \chi\circ\exp(t X_W)\circ\chi^{-1}$,
	where $X_U$, $\eps Z$ and $X_W$ are the vector fields on
	$\T^*\Smanifold$ defined respectively by the
	triples $(\mathsf{h},0,0)$, $(0,\mathsf{e},0)$ and $(0,0,W)$.
	From the definition of Algorithm~\ref{alg:3S} it now follows
	that~$\Phi^h_\mathrm{3S}=\chi\circ\Upsilon^h_\mathrm{3S}\circ\chi^{-1}$.
	
	Likewise, the vector field $X_V+\eps Y$ is defined by the triple $(0,\mathsf{d},V)$.
	Thus, it follows again from diagram~\eqref{eq:commutative_coordinate_independence} 
	that $\exp(t(X_V+\eps Y))=\chi\circ\exp(t(X_W+\eps Z))\circ\chi^{-1}$, and
	from the definition of Algorithm~\ref{alg:2S} we get
	$\Phi^h_\mathrm{2S}=\chi\circ\Upsilon^h_\mathrm{2S}\circ\chi^{-1}$.
	
	Lastly, notice that $X_V+\eps Y$ has trivial
	dynamics in~$q$. Thus, it reduces to a system
	on $\T^*_q Q$
	(only the $p$~variable is affected by its flow).
	From diagram~\eqref{eq:commutative_coordinate_independence}
	it follows that $X_W+\eps Z$ corresponds to a
	change of variables $q=\phi(s)$ and $r = \T^*_q\phi\cdot p$.
	This is a linear map in the
	momentum variable. 
	It follows that
	$\Psi^h_{X_V+\eps Y} = \chi\circ\Psi^h_{X_W+\eps Z}\circ\chi^{-1}$
	since $\Psi^h$ is a linear method, i.e., invariant under
	a linear change of variables (every Runge-Kutta method is).
	It now follows that 
	$\Phi^h_\mathrm{RKS}=\chi\circ\Upsilon^h_\mathrm{RKS}\circ\chi^{-1}$
	from the definition of Algorithm~\ref{alg:RKS}.
\end{proof}

\begin{remark} \label{rem:configuration_coordinate_independence}
	Although invariance with respect to choice of configuration 
	coordinates for numerical integration of system~\eqref{eq:perturbed_Hamiltonian_system} 
	seems to be a highly natural requirement, most families of integrators
	do not have this property. For example, symplectic Runge-Kutta methods, which are
	invariant under a general linear change of coordinates and preserve 
	the canonical symplectic form, 
	are not invariant under co-tangent lifted point transformations.
	In engineering application, this has an impact, since the result of 
	a numerical simulation then depends on whether Cartesian, cylindrical
	or spherical coordinates are used.
\end{remark}


\section{Backward Error Analysis} 
\label{sec:backward_error_analysis}

In this section we present a backward error analysis for the
methods presented in Section~\ref{sec:numerical_integration_schemes}.
We start off with a brief review of the general notion of backward error analysis
for ordinary differential equations. 
Thereafter, we show how the framework
can be used for the class of Rayleigh damped problems
considered in this paper.

\subsection{Review of Framework} 
\label{sub:review_of_framework_rev_2}

Let $\P$ be a phase space manifold and $X\in\Xcal(\P)$ a smooth vector field.
The set of diffeomorphisms on~$\P$ is denoted $\Diff(\P)$
(this set forms an infinite dimensional group under compositions).
Following previous authors~\cite{Re1999,Mo2009thesis,Ha2011},
we now define rigorously what is meant by an integrator.

\begin{definition}\label{def:integrator}
	An \emph{integrator} for a vector field $X\in\Xcal(\P)$ is a one-parameter family
	$\Phi^h:\P\to\P$ of diffeomorphisms that is smooth in $h$
	and satisfies: $\Phi^0 = \Id$ and
	$\Phi^h(z) - \exp(h X)(z) = \mathcal{O}(h^{r+1})$ for all $z\in\P$, 
	where $r\geq 1$ is the \emph{order}
	of the integrator.
	%
\end{definition}

Given an integrator $\Phi^h$ for $X\in\Xcal(\P)$,
the basic notion of backward error analysis is to
find a \emph{modified vector field} $X_h$, depending 
smoothly on~$h$, such that its flow coincides with $\Phi^h$, i.e.,
such that $\exp(h X_h) = \Phi^h$. Thus, the question of finding
a modified vector field is equivalent to the question of finding
an inverse of the exponential map $\exp:\Xcal(\P)\to\Diff(\P)$.
However, it is well known that $\exp$ is \emph{not} surjective,
not even in an arbitrary small neighborhood of the identity (see e.g.~\cite{Gr1988}).
Thus, consider instead the restriction
to the real analytic case.
Indeed, let $\P$ be a real analytic manifold (cf.~\cite{KrMi1997b}), 
let $\Xcal^a(\P)\subset\Xcal(\P)$ be the
space of real analytic vector fields on $\P$, and let $\Diff^a(\P)\subset\Diff(\P)$
be the corresponding real analytic diffeomorphisms.
Even in this case, the restricted exponential map $\exp:\Xcal^a\to\Diff^a(\P)$
is \emph{not} locally surjective~\cite[Chap.~IX]{KrMi1997b}.
However, one can still find a \emph{formal modified vector field}
by asymptotic expansion in the step size parameter:
\begin{equation}\label{eq:formal_bea_series}
	X_h = X_0 + h X_1 + h^{2} X_{2} + \ldots
\end{equation}
where $X_0,X_{1},X_{2},\ldots \in \Xcal^a(\P)$. Given an integrator $\Phi^h$
for~$X\in\Xcal^a(\P)$ the vector fields $X_0,X_1,X_2,\ldots$ 
are defined recursively by
\[
	X_{k+1}(z) = \lim_{h\to 0} \frac{\Phi^h(z)-\exp(h X_{h,k})(z)}{h^{k+1}}
\]
where $X_0 = X$ and $X_{h,k} = \sum_{i=0}^k h^i X_i$.
Thus, with this construction $\Phi^h(z) - \exp(h X_{h,k})(z) = \mathcal{O}(h^{k+1})$,
which follows from Taylor's theorem.
Notice that if the integrator is of order~$r$, then $X_k =0$ for $0 < k < r$.

In general the formal series~\eqref{eq:formal_bea_series} does not converge,
so instead one has to truncate the series.
The basic result in backward error analysis states that there is an optimal
truncation index $k$, depending on $h$,
such that $\Phi^h - \exp(h X_{h,k})$ is exponentially small. There are many contributors
to this result: we refer to~\cite[Chap.~IX]{HaLuWa2006} and references therein 
for an account of its origin. The original result is obtained for Euclidean
phase space $\P=\realset^n$. For the purpose of this paper, we use a
generalization to manifolds given in~\cite{Ha2011}:

\begin{theorem}\label{thm:bea_main_result}
	Let $\P$ be a real analytic manifold, $\set U\subset \P$ a compact subset,
	$X\in\Xcal^a(\P)$ a real analytic vector field, and $\Phi^h$
	an integrator for $X$ such that $h\mapsto \Phi^h(z)$ is real analytic
	for all $z\in\set U$.
	Then there exists a distance function $\mathrm{dist}:\P\times\P\to\realset^+$,
	a truncation index $k$ (depending on $h$),
	and constants $C,\gamma,h_0>0$ such that
	\[
		\mathrm{dist}\big(\Phi^h(z),\exp(h X_{h,k})(z)\big) \leq h C \e^{-\gamma/h}
	\]
	for all $z\in\set U$ whenever $h\leq h_0$.
\end{theorem}

\begin{remark}
	The distance function used in Theorem~\ref{thm:bea_main_result} is obtained
	by first embedding $\P$ in $\realset^n$ (by the Whitney embedding theorem)
	and then restricting the Euclidean distance function on $\realset^n$ to the embedded
	submanifold. In particular, in the case when $\P=\realset^n$, the distance function
	used is just the usual Euclidean distance. See~\cite{Ha2011} for details.
\end{remark}

Let us now consider the special case when $\P=\coTQ$ (as in this paper).
Denote by $\XcalHam(\coTQ)$ the space of Hamiltonian vector fields
on $\coTQ$ (with respect to the canonical symplectic form), i.e.,
$X\in\XcalHam(\coTQ)$ implies that $X=X_H$ for some Hamiltonian function
$H\in\Fcal(\coTQ)$. Correspondingly, we have the subgroup
of exact symplectic diffeomorphisms $\DiffHam(\coTQ)$.
It is well known that $\exp(X)\in\DiffHam(\coTQ)$ if and only
if $X\in\XcalHam(\coTQ)$.

\begin{remark}
	Notice that there is, in general, a difference between
	\emph{Hamiltonian vector fields} $\XcalHam(\coTQ)$
	and \emph{symplectic vector fields} $\XcalSp(\coTQ)$,
	the former being a subalgebra of the latter.
	Indeed, $X\in\XcalSp(\coTQ)$ means that $X$
	preserves the symplectic form, whereas $X\in\XcalHam(\coTQ)$
	means that there exists a globally defined $H\in\Fcal(\coTQ)$
	such that $X=X_H$. Correspondingly, the group of 
	\emph{exact symplectic maps} $\DiffHam(\coTQ)$ is
	a subgroup of the group of \emph{symplectic maps}~$\DiffSp(\coTQ)$.
	See e.g.~\cite{KrMi1997b,MoPeMaMc2011} for further reading.
\end{remark}

Now, let $\Phi^h$ be an integrator for $X_H\in\XcalHam(\coTQ)$.
Naturally, $\Phi^h$ is called a \emph{symplectic integrator}
if $\Phi^h\in\DiffSp(\coTQ)$ for each fixed~$h>0$, and it is called an
\emph{exact symplectic integrator} if $\Phi^h \in \DiffHam(\coTQ)$ for each fixed~$h>0$. 
In addition to
Theorem~\ref{thm:bea_main_result}, another basic
result in backward error analysis is that if $\Phi^h$ is an
exact symplectic integrator, then each of the vector fields in the
formal series~\eqref{eq:formal_bea_series} belong to~$\XcalHam(\coTQ)$.
Thus, each of these vector fields correspond to a Hamiltonian function.
For a real analytic Hamiltonian system integrated by an exact symplectic
integrator, application of Theorem~\ref{thm:bea_main_result} yields
that if the numerical solution stays on a compact subset, then
there exists a truncated 
modified Hamiltonian function (corresponding to a truncated modified
Hamiltonian vector field) which is conserved by the numerical solution up
to an exponentially small term for exponentially long times
(see e.g.~\cite[Sect.~IX.8]{HaLuWa2006}). Furthermore, it can
be shown that the family of truncated modified Hamiltonian functions
have a common Lipschitz constant, independent of the 
step size~\cite[Sect.~IX.7.2]{HaLuWa2006}.

A deeper consequence of the modified vector field preserving the
Hamiltonian structure is that
if $X_H$ is \emph{Arnold-Liouville integrable} (cf.~\cite[Sect.~X.1]{HaLuWa2006}), 
i.e., it has invariant tori,
then an exact symplectic integrator will preserve perturbed tori for
exponentially long times
(under certain technical assumptions associated with KAM theory).


\subsection{Application to Perturbed Hamiltonian Problems} 
\label{sub:application_to_perturbed_hamiltonian_problems_rev2}

In this section we derive the modified vector fields for
the numerical methods suggested in Section~\ref{sec:numerical_integration_schemes}.
Our aim is to study the energy dissipation behavior of the methods,
using a backward error analysis approach. The idea is to study the
evolution in the numerical solution
of a truncated modified Hamiltonian function.
In particular, we show that up to an exponentially small term,
the suggested methods have asymptotically correct energy behavior
in the dissipation parameter~$\eps$. Furthermore, the evolution of a
truncated modified Hamiltonian
for Algorithm~\ref{alg:3S}, is monotone up to an exponentially small term.

For convenience, we say that a numerical method for
perturbed Hamiltonian systems of the
form~\eqref{eq:perturbed_Hamiltonian_system} is
(exact) symplectic, if it is (exact) symplectic for~$\eps=0$.
Thus, the suggested methods Algorithm~\ref{alg:3S}--\ref{alg:RKS}
are exact symplectic, which follows directly from the splitting
approach, and the fact that $\DiffHam(\coTQ)$ forms a group under composition.
Throughout this section, we assume that phase space~$\coTQ$
carries a real analytic structure.

%

Let $Z_\eps = X_T + X_V + \eps Y$ be the vector field 
in equation~\ref{eq:perturbed_Hamiltonian_system_condense}. From
equation~\ref{eq:exact_energy_dissipation_rate} it follows that
\begin{equation}\label{eq:exact_energy_dissipation_rate_integral_form}
	H\big(\!\exp(t Z_\eps)(z)\big)-H(z) = -\eps \int_{0}^t \mathsf{d}_{q(s)}(p(s)^\sharp,p(s)^\sharp)\ud s
\end{equation}
where $(q(s),p(s)) = \exp(s Z_\eps)(z)$. 
Our aim is to find a modified analog 
of~\eqref{eq:exact_energy_dissipation_rate_integral_form}
corresponding a numerical integrator, ideally by
replacing the Hamiltonian with a modified Hamiltonian,
and the dissipation tensor with a modified dissipation tensor.
In particular, there are two qualitative features of
equation~\eqref{eq:exact_energy_dissipation_rate_integral_form}
that we would like to find analogs of in the numerical solution:
\begin{itemize}
	\item $H\big(\!\exp(t Z_\eps)(z)\big)-H(z)$ is proportional to $\eps$;
	\item $H\big(\!\exp(t Z_\eps)(z)\big)-H(z) \leq 0$
	for $t\geq 0$.
\end{itemize}

We first give the general structure of the formal modified vector field
corresponding to a symplectic method:

\begin{lemma}\label{lem:formal_modified_vectorfield_general_symplectic}
	Let $\Phi^h$ be an exact symplectic integrator of order~$r$ 
	for problem~\eqref{eq:perturbed_Hamiltonian_system}.
	Then its formal modified vector field is of the form
	\[
		Z_{h,\eps} = X_T + X_V + \eps Y + \sum_{k=r}^\infty h^k (X_{H_k} + \eps Y_{k,\eps})
	\]
	where $X_{H_k}\in\XcalHam(\coTQ)$ and $Y_{k,\eps}\in\Xcal(\coTQ)$ depending
	smoothly on~$\eps$.
\end{lemma}

\begin{proof}
	For $\eps=0$ the formal series is of the form
	\[
		Z_{h,0} = X_T + X_V  + \sum_{k=r}^\infty h^k X_{H_k} ,
	\]
	since the method is exact symplectic.
	The result now follows from Taylor's theorem 
	since $\Phi^h$ depends smoothly on~$\eps$.
\end{proof}

From this result we get a modified analog of
equation~\eqref{eq:exact_energy_dissipation_rate_integral_form}
for general exact symplectic integrators. Indeed, let
$H_{h,N} = T + V + \sum_{k=r}^Nh^k H_k$, with $H_k$ as in
Lemma~\ref{lem:formal_modified_vectorfield_general_symplectic}. Then we
have the following result:

\begin{theorem}\label{thm:energy_dissipation_rate_general_integrator}
	Let $\Phi^h$ be an exact symplectic integrator of order~$r$ for 
	problem~\eqref{eq:perturbed_Hamiltonian_system},
	and let $\set U\subset\coTQ$ be a compact subset.
	Assume that $Z_{\eps}=X_T+X_V+\eps Y$ is real analytic in $\set U$,
	and that $h\mapsto\Phi^h(z)$ is real analytic for every $z\in\set U$.
	Then there is a truncation index $N$ depending on $h$ and $\eps$ 
	such that
	\[
		\Big| H_{h,N}(\Phi^h(z)) - H_{h,N}(z) + \eps \int_0^{h} 
			\mathsf{d}_{h,N,\eps}\big(\!\exp(s Z_{h,N,\eps})(z)\big)\ud s \Big|
		\leq h \lambda C \e^{-\gamma/h}
	\]
	whenever $h\leq h_0$ and $\eps\leq \eps_0$, where $\lambda,C,\gamma,h_0,\eps_0>0$ 
	are constants not depending on~$h$ or~$\eps$, and where
	\[
		\mathsf{d}_{h,N,\eps}(z) := \mathsf{d}_{q}(p^\sharp,p^\sharp) -
		\sum_{k=r}^N h^k\pair{\!\ud H_{h,N}(z),Y_{k,\eps}(z)},
	\]
	for $Y_{k,\eps}$ as in Lemma~\ref{lem:formal_modified_vectorfield_general_symplectic}.
\end{theorem}

\begin{proof}
	From truncation of the modified vector field 
	in Lemma~\ref{lem:formal_modified_vectorfield_general_symplectic}
	it follows that
	\[
		\begin{split}
			\LieD{Z_{h,N,\eps}}H_{h,N}(z) &= \pair{\!\ud H_{h,N}(z),Z_{h,N,\eps}(z)}
			\\
			&= \pair{\!\ud H_{h,N}(z), X_{H_{h,N}(z)} + \eps Y_{h,N,\eps}(z) }
			\\
			&= \eps \pair{\!\ud H_{h,N}(z),Y_{h,N,\eps}(z)}
			\\
			&= -\eps \big( 
				\underbrace{
				\mathsf{d}_q(p^\sharp,p^\sharp) -
				\sum_{k=r}^N h^k\pair{\!\ud H_{h,N}(z),Y_{k,\eps}(z)}
				}_{\mathsf{d}_{h,N,\eps}(z)}
				\big) .
		\end{split}
	\]
	Thus, 
	\begin{equation*}\label{eq:modH_evolution}
		H_{h,N}\big(\!\exp(t Z_{h,N,\eps})(z)\big)-
		H_{h,N}(z) = -\eps\int_0^t \mathsf{d}_{h,N,\eps}\big(\!\exp(s Z_{h,N,\eps})(z)\big)\ud s .
	\end{equation*}
	Moreover, we have
	\[
		\begin{split}
			H_{h,N}\big( \Phi^h(z)\big) - H_{h,N}(z)
			&=
			H_{h,N}\big(\!\exp(t Z_{h,N,\eps})(z)\big) - H_{h,N}(z)
			\\
			&+
			H_{h,N}\big( \Phi^h(z)\big) - H_{h,N}\big(\!\exp(t Z_{h,N,\eps})(z)\big),
		\end{split}
	\]
	so we get
	\[
		\begin{split}
			\big| 
				H_{h,N}\big( \Phi^h(z)\big) - H_{h,N}(z)
				+ \eps\int_0^t \mathsf{d}_{h,N,\eps}\big(\!\exp(s Z_{h,N,\eps})(z)\big)\ud s
			\big| 
			\\
			=
			\big|
				H_{h,N}\big( \Phi^h(z)\big) - H_{h,N}\big(\!\exp(t Z_{h,N,\eps})(z)\big)
			\big|
		\end{split}
	\]
	%
	%
	Let $\lambda>0$ be an $h$--independent Lipschitz constant for $H_{h,N}$
	with respect to the distance function~$\mathrm{dist}$ in 
	Theorem~\ref{thm:bea_main_result}
	(such a Lipschitz bound is known to exists for truncated modified Hamiltonians,
	see e.g.~\cite[Theorem~8.1, Sect.~IX.8]{HaLuWa2006}).
	Thus, an estimate for the right hand side is
	\begin{equation}\label{eq:H_estimate}
		\begin{split}
			\big|
				H_{h,N}\big(\Phi^h(z)\big) -H_{h,N}\big(\!\exp(n h Z_{h,N,\eps})(z)\big)
			\big|
			&
			\leq \lambda \, \mathrm{dist}\big(\Phi^h(z),\exp(h Z_{h,N,\eps})(z)\big) .
		\end{split}
	\end{equation}
	Now, in order to get a bound which is uniform in~$\eps$, let $\eps_0>0$
	and consider the extended map $\hat\Phi^h(z,\eps) := (\Phi^h(z),\eps)$ on $\set U \times [0,\eps_0]$
	(notice that $\Phi^h$ depends on~$\eps$), which is an integrator
	for the extended vector field $\hat Z(z,\eps) := (Z_\eps(z),0)$.
	Since the extended $\eps$--part is integrated exactly by $\hat\Phi^h$,
	it holds that $\exp(h \hat Z_{h,N})(z,\eps) = (\exp(h Z_{h,N,\eps})(z),\eps)$, so we have
	\[
		\widehat{\mathrm{dist}}\big(\hat\Phi^h(z,\eps), \exp(h \hat Z_{h,N})(z,\eps)\big) 
		= \Big(\mathrm{dist}\big(\Phi^h(z),\exp(h Z_{h,N,\eps})(z)\big)^2 + \abs{\eps-\eps}^2\Big)^{\frac{1}{2}}.
	\]
	Thus, in combination with~\eqref{eq:H_estimate} we have
	\[
		\begin{split}
			\big|
				H_{h,N}\big(\Phi^h(z)\big) -H_{h,N}\big(\!\exp(n h Z_{h,N,\eps})(z)\big)
			\big|
			&
			\leq \lambda \, \widehat{\mathrm{dist}}\big(\hat\Phi^h(z,\eps), \exp(h \hat Z_{h,N})(z,\eps)\big)
			\\
			& \leq h \lambda C \e^{-\gamma/h}
		\end{split}
	\]
	for $h\leq h_{0} > 0$, where the last estimate follows from Theorem~\ref{thm:bea_main_result}
	applied to the extended integrator $\hat\Phi^h$, which can be done since $\set U\times [0,\eps_0]$
	is compact and $\hat Z(z,\eps)$ is real analytic on $\set U\times [0,\eps_0]$.
	This proves the theorem.
\end{proof}

From the result we see that, up to an exponentially small term, the modified
energy evolves with a rate which is $\mathcal{O}(\eps h^r)$ close
to the energy dissipation rate of the exact solution.
Thus, the dynamics is asymptotically correct in
the perturbation parameter~$\eps$.

In transient regions of the phase space, i.e., regions where
the value of dissipation tensor $\mathsf{d}$ is relatively large,
Theorem~\ref{thm:energy_dissipation_rate_general_integrator} asserts
that the energy evolves in a monotone fashion. More precisely,
let $\set U\subset\coTQ$ be a compact subset such that
$\sup_{z\in\set U}
\abs{\mathsf{d}_{h,N,\eps}(z)-\mathsf{d}_q(p^\sharp,p^\sharp)}/\mathsf{d}_q(p^\sharp,p^\sharp) 
\ll 1$
for the chosen step size.
As long as the numerical solution stays in $\set U$,
Theorem~\ref{thm:energy_dissipation_rate_general_integrator} then
asserts that the modified Hamiltonian is
decreasing from step to step, and that the numerical dissipation
rate is
accurate relative to the exact dissipation rate.
Thus, the transient phase in the simulation is captured
very well with a symplectic integrator, which is often important
in application, for example when accurate estimates of energy losses are essential.

After the transient phase 
some type of ``steady state'' 
dynamics is commonly reached. 
In phase space, this
often occur close to the subset
$\ker \mathsf{d} = \{ z\in\coTQ ; \mathsf{d}_{q}(p^\sharp,p^\sharp)=0\}$, i.e.,
where the dissipation vanishes. Furthermore, it is common that
dissipative forces (e.g.~friction forces) are small relative
to conservative forces. In such regions,
one does not necessarily have that $\mathsf{d}_{h,N,\eps}(z)\geq 0$, so
Theorem~\ref{thm:energy_dissipation_rate_general_integrator}
does not assert that the evolution of the modified Hamiltonian is monotone.
However, for the special case of the suggested Algorithm~\ref{alg:3S},
we now derive a refined version of 
Theorem~\ref{thm:energy_dissipation_rate_general_integrator},
for which monotone decrease of the modified Hamiltonian is asserted
for small enough step sizes (up to an exponentially small correction term).
%
%
The result is based on the following two lemmas:

\begin{lemma}\label{lem:even_order_modified_hamiltonian_rev2}
	The formal modified Hamiltonian $H_h$ for the
	Störmer-Verlet method $\PhiSV^h$ 
	applied to problem~\eqref{eq:perturbed_Hamiltonian_system_condense}
	with $\eps=0$
	is of the form 
	\[
		\begin{split}
			H_h(q,p) = \mathsf{g}_q(p^\sharp,p^\sharp) + V(q) 
			& + \sum_{\ell=1}^\infty h^{2\ell}\,V_{2 \ell}(q) \\
			& + \sum_{k=1}^\infty
			\sum_{\ell = 1}^\infty h^{2(k+\ell-1)} 
			(\mathsf{g}_{2k,2\ell})_q(\underbrace{p^\sharp,\ldots,p^\sharp}_{\text{$2k$ copies}})			
		\end{split}
		%
	\]
	where $V_{2\ell} \in \Fcal(\coTQ)$ and 
	$\mathsf{g}_{2k,2\ell}\in\mathcal{T}_{2k}^0(\Q)$ are a totally
	symmetric tensor fields.
	%
	%
\end{lemma}

\begin{proof}
	If $\mathsf{a}\in\mathcal{T}_{k}^0(\Q)$ and $\mathsf{b}\in\mathcal{T}_{\ell}^0(\Q)$
	are two totally symmetric tensor fields, then the Poisson bracket
	between the functions $(q,p)\mapsto\mathsf{a}_q(p^\sharp,\ldots,p^\sharp)$ and
	$(q,p)\mapsto\mathsf{b}_q(p^\sharp,\ldots,p^\sharp)$ is given by
	$(q,p)\mapsto\mathsf{c}_q(p^\sharp,\ldots,p^\sharp)$, where $\mathsf{c}\in\mathcal{T}_{k+\ell}^0(\Q)$
	is the totally symmetric tensor 
	\[
		\begin{split}
			\mathsf{c}_q(\cdot,\ldots,\cdot) &=
			\mathsf{a}_q(\pd_q\mathsf{b}_q(\cdot,\ldots,\cdot),\cdot,\ldots,\cdot)
			+ \ldots +
			\mathsf{a}_q(\cdot,\ldots,\cdot,\pd_q\mathsf{b}_q(\cdot,\ldots,\cdot))\\
			- &
			\mathsf{b}_q(\pd_q\mathsf{a}_q(\cdot,\ldots,\cdot),\cdot,\ldots,\cdot)
			+ \ldots +
			\mathsf{b}_q(\cdot,\ldots,\cdot,\pd_q\mathsf{a}_q(\cdot,\ldots,\cdot)).
		\end{split}
	\]
	Since Störmer-Verlet is a splitting method, and since $V$ and $\mathsf{g}$
	are tensors, it follows from the Baker-Campbell-Hausdorff formula
	that the modified Hamiltonian is a sum of terms of the form
	$h^{i+j}\mathsf{g}_{i,j}(p^\sharp,\ldots,p^\sharp)$,
	where $\mathsf{g}_{i,j}\in\mathcal{T}_i^0(\Q)$.
	Further, since the Störmer-Verlet method preserves
	reversibility 
	it holds that $H_h(q,p) = H_h(q,-p)$. Thus, only the
	even order tensors remain. Also, only the even powers of $h$
	survive, since the method is symmetric. This yields the result.
	%
	%
\end{proof}

\begin{lemma}\label{lem:diss_mod_Ham_for_SV}
	Let $Y\in\Xcal(\coTQ)$ be the dissipative vector field in 
	equation~\eqref{eq:perturbed_Hamiltonian_system_condense}, and
	let $H_{h,N}$ be the truncation of the formal modified Hamiltonian in
	Lemma~\ref{sub:application_to_perturbed_hamiltonian_problems_rev2}.
	Further, let $\set U\subset \coTQ$ be a compact subset.
	Then there exists $h_0>0$ such that 
	\[
		\LieD{Y}H_{h,N} = - \mathsf{d}_q(E_{h,N}(z)p^\sharp,p^\sharp) \leq 0
	\]
	for all $z\in\set U$ whenever $h\leq h_0$,
	where $E_{h,N}(z):\T_q\Q\to\T_q\Q$ is an operator which is self-adjoint and
	positive definite with respect to $\mathsf{d}$, and
	such that $E_{h,N}(z) - \Id = \mathcal{O}(h^2)$ for all $z\in\set U$.
	%
	%
\end{lemma}

\begin{proof}
	Let $\inner{\cdot,\cdot}$ denote the metric $\mathsf{g}_q(\cdot,\cdot)$.
	Let $C(q):\T_q\Q\to\T_q\Q$ be the positive 
	semi-definite self-adjoint operator defined by
	$\inner{C(q)u,v}=\mathsf{d}_q(u,v)$.
	In other words, $C(q)u = \mathsf{d}_q(u,\cdot)^\sharp$.
	
	It holds that
	\[
		\begin{split}
			\LieD{Y}({\mathsf{g}}_{2k,2\ell})_q(p^\sharp,\ldots,p^\sharp) =\; &
			({\mathsf{g}}_{2k,2\ell})_q\big(C(q) p^\sharp,p^\sharp,\ldots,p^\sharp\big)
			+ \\ &
			({\mathsf{g}}_{2k,2\ell})_q\big(p^\sharp,C(q) p^\sharp,p^\sharp,\ldots,p^\sharp\big)
			+\ldots + \\ &
			({\mathsf{g}}_{2k,2\ell})_q\big(p^\sharp,\ldots,p^\sharp,C(q) p^\sharp\big)
			=: (\mathsf{d}_{2k,2\ell})_q(p^\sharp,\ldots,p^\sharp),
		\end{split}
	\]
	which defines the totally symmetric tensor $\mathsf{d}_{2k,2\ell}\in\mathcal{T}_{2k}^0(\coTQ)$.
	Notice that the kernel of $\mathsf{d}_{2k,2\ell}$ is contained in the kernel of $C(q)$,
	i.e., if $u\in\ker(C(q))$, then $(\mathsf{d}_{2k,2\ell})_q(u,\cdot,\ldots,\cdot)=0$.
	Next, we define the self-adjoint operator $C_{2k,2\ell}(z):\T_q\Q\to\T_q\Q$ by
	$\inner{C_{2k,2\ell}(z)u,v}=(\mathsf{d}_{2k,2\ell})_q(p^\sharp,\ldots,p^\sharp,u,v)$.
	Again, notice that $\ker(C_{2k,2\ell}(z))\subset\ker(C(q))$.
	All together, summing up the terms 
	up to order~$h^N$
	in Lemma~\ref{lem:even_order_modified_hamiltonian_rev2},
	we get
	\[
		\LieD{Y}H_{h,N}(q,p) =
			-\inner{C(q)p^\sharp,p^\sharp} -
			h^2\inner{\mybreve{C}_{h,N}(z)p^\sharp,p^\sharp}
	\]
	where $\mybreve{C}_{h,N}(z)$ is a self-adjoint operator $\T_q\Q\to\T_q\Q$,
	with $\ker(\mybreve{C}_{h,N}(z))\subset\ker(C(q))$.
	Since both $C(q)$ and $\mybreve{C}_{h,N}(z)$ are self-adjoint 
	it holds that $(\ker C(q))^\bot$ is an invariant subspace for both of them 
	(follows from the spectral theorem). 
	Since~$\mathsf{d}$ has constant rank and since
	$\set U$ is compact, it follows that 
	\[
		\inf_{z\in\set{U}}\norm{C(q)|_{(\ker C(q))^\bot}}^2>0.
	\]
	Thus, there is an $h_0>0$ such that $C(q)+h^2\mybreve{C}_{h,N}(z)$ stays positive
	semi-definite when $h\leq h_0$ and $z\in\set U$; the zero eigenvalues stays zero, 
	and the perturbation,
	acting in $(\ker C(q))^\bot$, is small enough for its positive eigenvalues to stay positive.
	Hence,
	\[
		\LieD{Y}H_{h,N}(z) = -\inner{\big(C(q)+h^2\mybreve{C}_{h,N}(z)\big)p^\sharp,p^\sharp} \leq 0
	\]
	if $h\leq h_0$ and $z\in\set{U}$. Next, define $E_{h,N}(z):\T_q\Q\to\T_q\Q$
	by $\mathsf{d}_q(E_{h,N}(z) u,v) = \inner{\big(C(q)+h^2\mybreve{C}_{h,N}(z)\big)u,v}$
	for $u,v\in (\ker C(q))^\bot$ and $E_{h,N}(z)|_{\ker C(q)} = \Id$.
\end{proof}

We now give the refined result on evolution of the
modified Hamiltonian function for Algorithm~\ref{alg:3S}.

\begin{theorem}\label{thm:monotonicity_mod_Ham_alg3S}
	Let $\set U\subset\coTQ$ be a compact subset, and assume that
	$\mathsf{g},\mathsf{d},V$ are real analytic on $\set U$.
	Then, for the integrator $\PhiThreeS^h$, defined by Algorithm~\ref{alg:3S},
	there exists $h_0>0$, depending on $\set U,\mathsf{g},\mathsf{d},V$ but not
	on $\eps$, and a truncation index $N$ depending on $h$, such that 
	\[
		\begin{split}
			\big|
				H_{h,N} & (\PhiThreeS^h(z)) - H_{h,N}(z) 
				+ 
				\eps \int_{0}^{h/2}
				\mathsf{d}_{h,N} \big(\! 
					\exp(s \eps Y)(z)
				\big)
				\ud s
				\\
				& +
				\eps \int_{0}^{h/2}
				\mathsf{d}_{h,N}
				\big(\!
					\exp(s \eps Y)\circ\PhiSV^h\circ\exp(h \eps Y/2)(z)
				\big)
				\ud s
			\big|
			\leq h \lambda C \e^{-\gamma/h}
		\end{split}
	\]
	whenever $h\leq h_0$ and $z\in\set U$,
	where $\lambda,C,\gamma,h_0>0$  are constants
	not depending on~$\eps$, and where
	\[
		\mathsf{d}_{h,N}(z) := \mathsf{d}_q(E_{h,N}(z)p^\sharp,p^\sharp) \geq 0
	\]
	%
	with $E_{h,N}(z):\T_q\Q\to\T_q\Q$ as in Lemma~\ref{lem:diss_mod_Ham_for_SV}.
	%
\end{theorem}

\begin{proof}
	We have
	\[
		\begin{split}
			H_{h,N} & (\PhiThreeS^h(z)) - H_{h,N}(z)
			 = 
			 H_{h,N}\big(\!\exp(h\eps Y/2)(z)\big) - H_{h,N}(z)
			\\
			+ & H_{h,N}\big(\PhiSV^h\circ\exp(h\eps Y/2)(z)\big) - H_{h,N}\big(\!\exp(h\eps Y/2)(z)\big)
			\\
			+ & H_{h,N}\big(\!\exp(h\eps Y/2)\circ\PhiSV^h\circ\exp(h\eps Y/2)(z)\big) - 
				H_{h,N}\big(\!\PhiSV^h\circ\exp(h\eps Y/2)(z)\big)
		\end{split}
	\]
	From Theorem~\ref{thm:energy_dissipation_rate_general_integrator} with $\eps=0$
	it follows that
	\[
		H_{h,N}\big(\PhiSV^h\circ\exp(h\eps Y/2)(z)\big) - H_{h,N}\big(\!\exp(h\eps Y/2)(z)\big)
		\leq 
		h \lambda C \e^{-\gamma/h}
	\]
	whenever $h\leq h_0'$ for $\lambda,C,\gamma, h_0'$ independent of $\eps$ (since it is zero
	when we apply the theorem).
	Since $\exp(h \eps Y/2)$ is the exact $h$--flow of $\eps Y/2$ it follows
	from Lemma~\ref{lem:diss_mod_Ham_for_SV} that there exists $h_0''>0$
	(clearly independent of~$\eps$)
	such that $\mathsf{d}_{h,N}(z) = \mathsf{d}_q(E_{h,N}(z)p^\sharp,p^\sharp)\geq 0$.
	This gives
	the four remaining
	terms in the sum. Now chose $h_0 = \min(h_0',h_0'')$.
	%
	%
\end{proof}

Notice in particular that this result asserts, up to an exponentially small
rest term, that the modified Hamiltonian decreases monotonely for small
enough step sizes. The result is optimal, in the sense that
the same exponential term would remain if we take~$\eps=0$.

\subsection{Numerical Example: 2--D Pendulum} 
\label{sub:example_damped_pendulum}

\begin{figure} 
	\centering
	{\bf Phase diagrams for damped pendulum using Algorithms~\ref{alg:3S}--\ref{alg:RKS}} \\[2ex]
	\begin{tabular}{cc}
		$h=0.2$ & $h=0.3$ \\[-0.5ex]
		\includegraphics[width=0.4\textwidth]{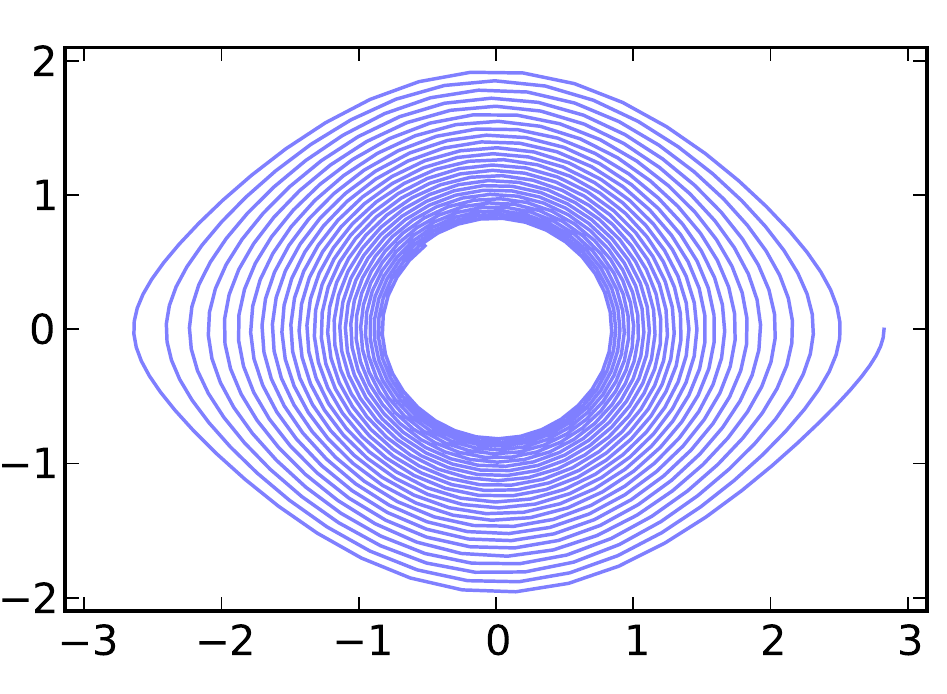} &
		\includegraphics[width=0.4\textwidth]{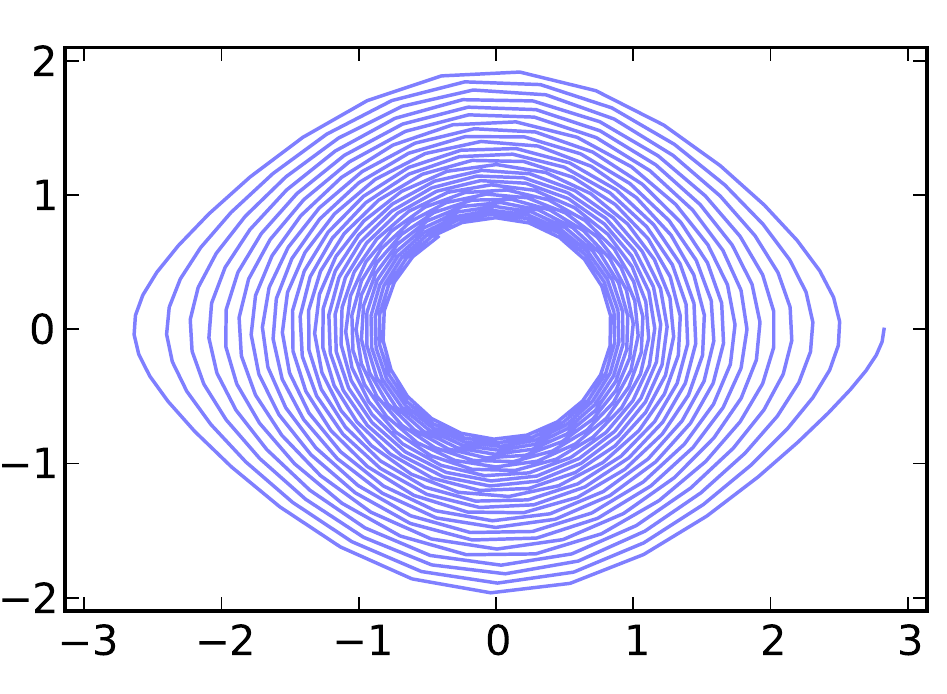} \\
		$h=0.4$ & $h=0.5$ \\[-0.5ex]
		\includegraphics[width=0.4\textwidth]{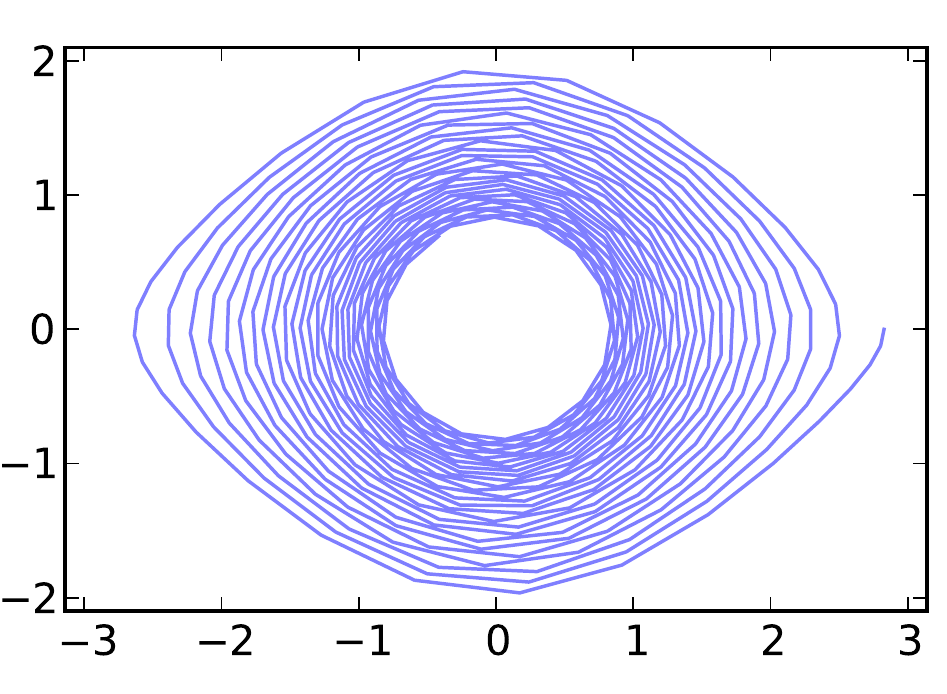} &
		\includegraphics[width=0.4\textwidth]{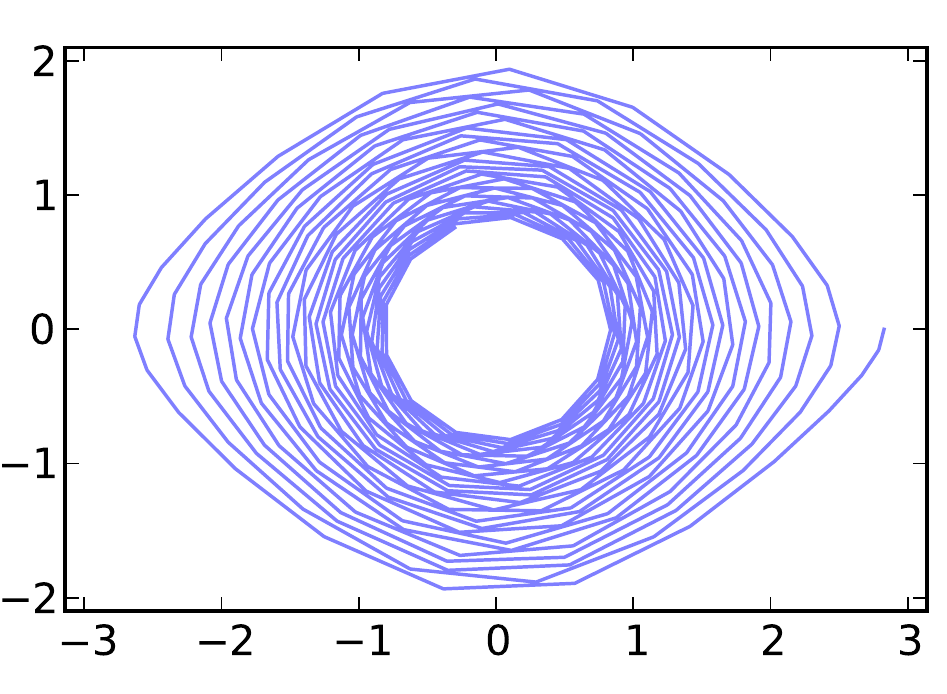}
	\end{tabular}
	\caption{Phase diagrams for the damped pendulum problems integrated
	with the suggested methods (they all overlap at this level of detail) 
	for various step sizes. 
	Notice that the amount of dissipation
	does not depend much on the step size
	(the size of the hole in the middle is roughly the same).}
	\label{fig:damped_pedulum_phase_sall}
\end{figure}

\begin{figure}
	\centering
	{\bf Phase diagrams for damped pendulum using Heun's method} \\[2ex]
	\begin{tabular}{cc}
		$h=0.2$ & $h=0.3$ \\[-0.5ex]
		\includegraphics[width=0.4\textwidth]{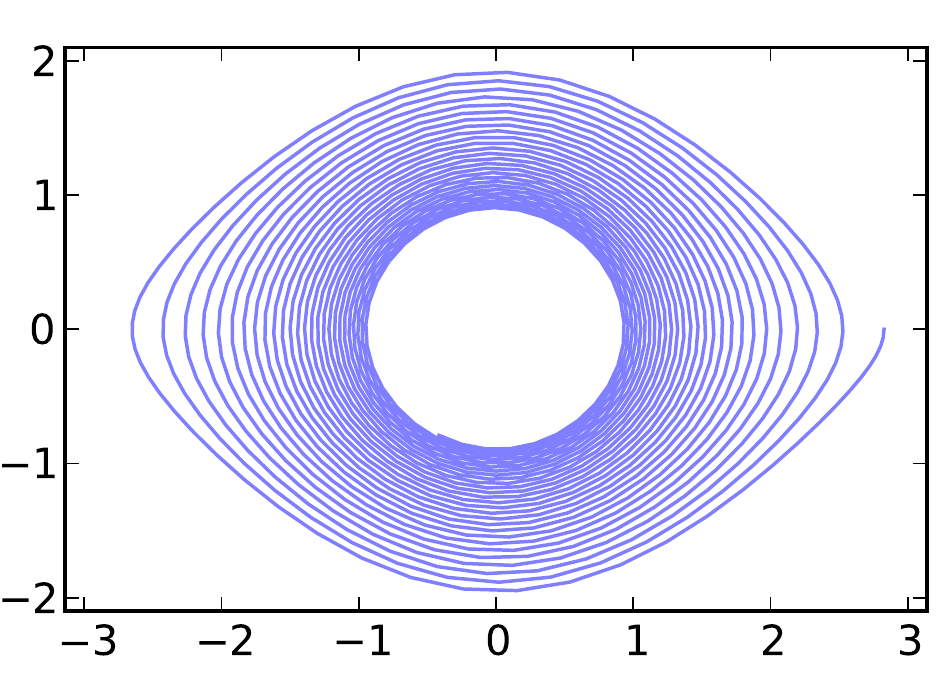} &
		\includegraphics[width=0.4\textwidth]{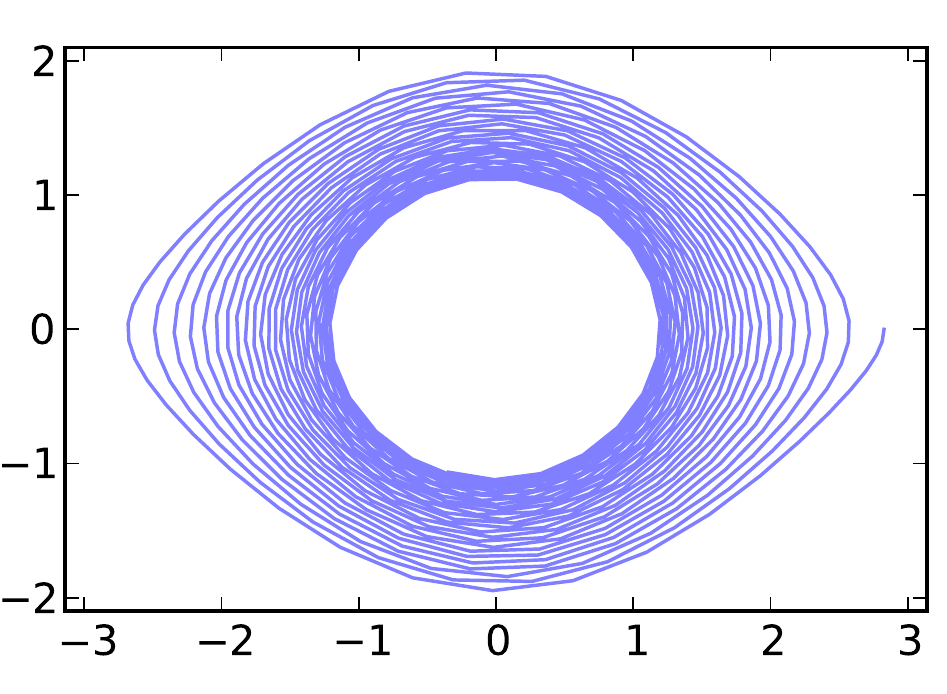} \\
		$h=0.4$ & $h=0.5$ \\[-0.5ex]
		\includegraphics[width=0.4\textwidth]{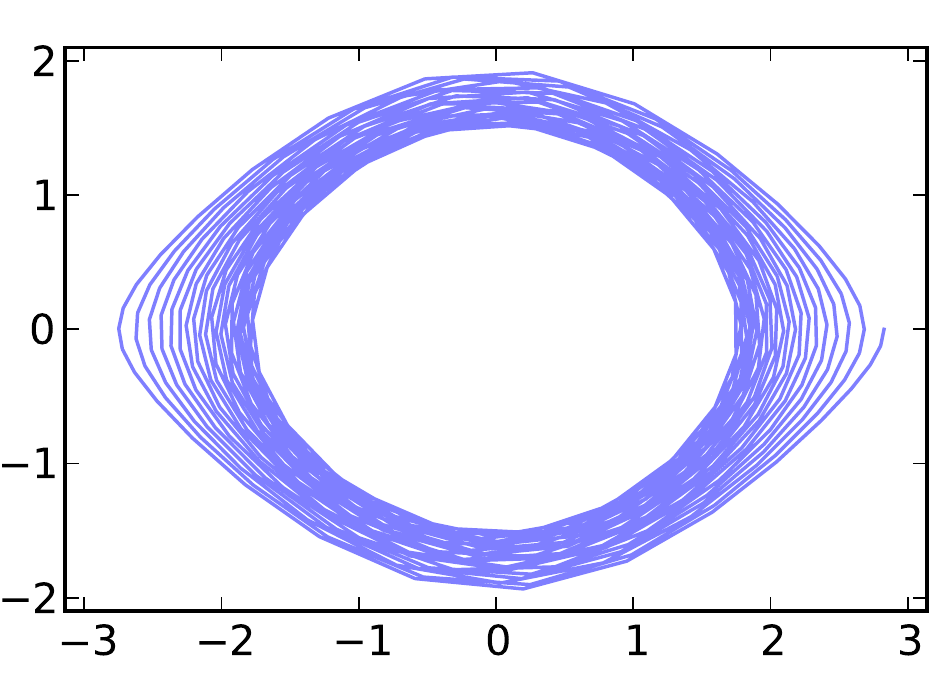} &
		\includegraphics[width=0.4\textwidth]{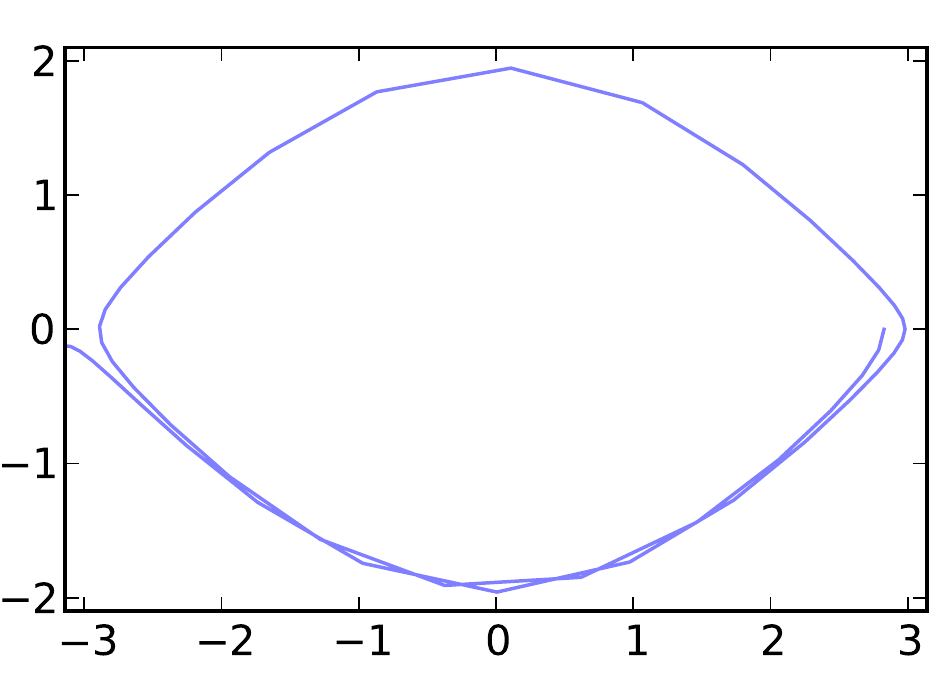}		
	\end{tabular}
	\caption{Phase diagrams for the damped pendulum problems integrated
	with Heun's method for various step sizes. Notice that the amount of dissipation
	(corresponding to the size of the hole in the middle)
	strongly depends on the step size.}
	\label{fig:damped_pedulum_phase_Heun}
\end{figure}

%
\begin{figure}
	\centering
	{\bf Evolution of energy error ($h=0.2$)} \\
	\hspace{-2ex}\raisebox{10ex}{$H_{\it err}$}
	\includegraphics[width=0.8\textwidth]{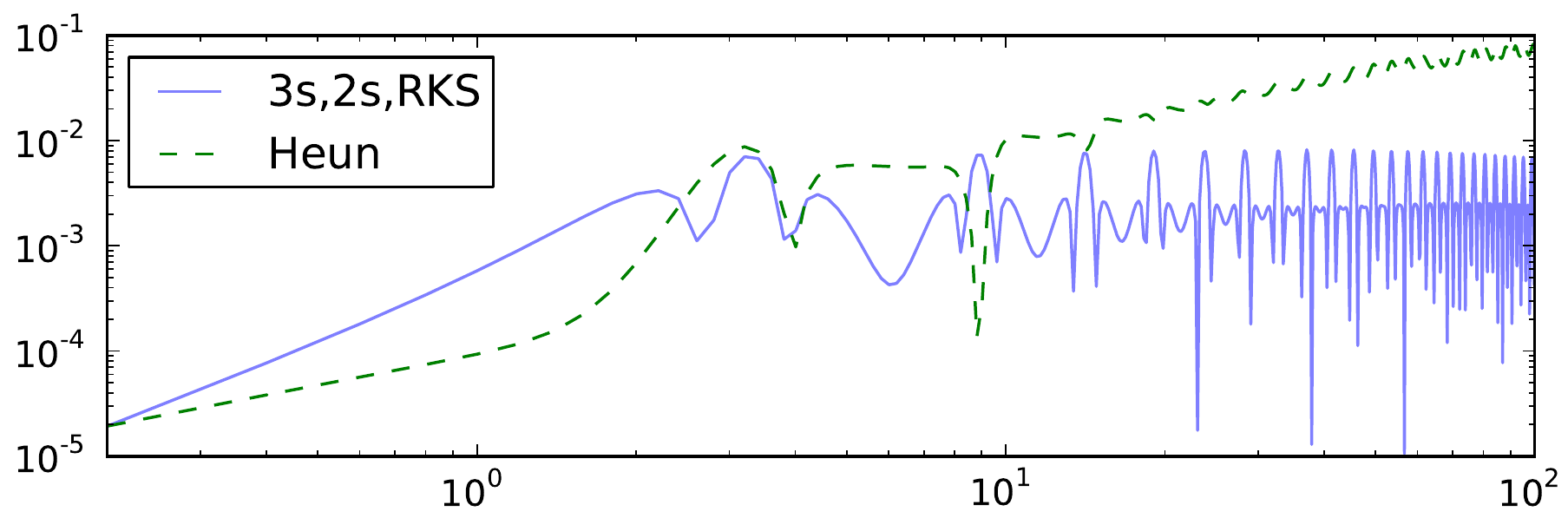}
	\\[-3ex] $t$ \\[2ex]
	{\bf Comparison of energy in between methods ($h=0.2$)} \\
	\hspace{-2ex}\raisebox{16ex}{$H_{\it err}$}
	\includegraphics[width=0.8\textwidth]{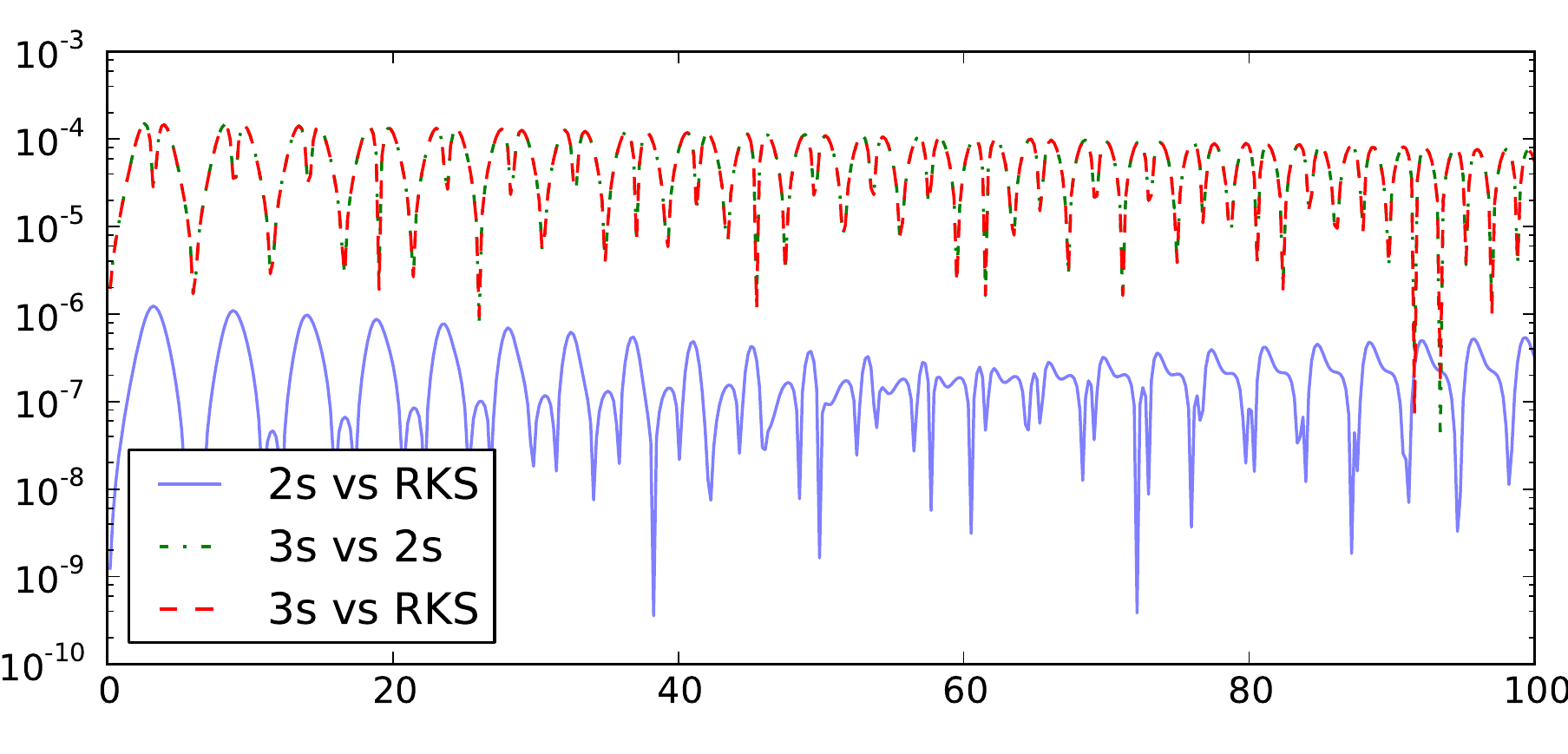}	
	\\[-3ex] $t$ \\
	\caption{
	(Upper) Evolution of energy for Algorithms~\ref{alg:3S}--\ref{alg:RKS},
	and for Heun's method. The error is bounded for
	the suggested methods, whereas it grows linearly for Heun's method.
	{(Lower)} Comparison of energy evolution in between 
	Algorithms~\ref{alg:3S}--\ref{alg:RKS}. The difference 
	is about a factor $10^2$ or $10^4$ smaller than
	the actual energy error.}
	\label{fig:damped_pendulum_energy_error}
\end{figure}
%
%

Let $\Q=\realset$, $\mathsf{g}_q(u,v)=\frac{1}{2}u v$, $\mathsf{d}=\mathsf{g}$
and $V(q) = 1-\cos(q)$. The system so obtained describes a damped non-linear pendulum
in the vertical plane, with unit mass and unit gravity
(see e.g.~\cite[Sect.~I.1]{HaLuWa2006} for details on this problem,
in the un-damped case).

The system is integrated with each of Algorithm~\ref{alg:3S}--\ref{alg:RKS},
as well as with Heun's explicit second order method.
(We also use Heun's method as the choice of Runge-Kutta method
for Algorithm~\ref{alg:RKS}.)
The following data was used: $q(0)=0.9\,\pi$, $p(0)=0$,
$\eps=10^{-2}$ and integration time interval~$[0,200]$,
for various step sizes $h\in[10^{-1},1]$.

Phase diagrams for various step sizes using Algorithms~\ref{alg:3S}--\ref{alg:RKS} are given
in Figure~\ref{fig:damped_pedulum_phase_sall}. 
There is virtually no difference between the suggested methods, i.e.,
the phase curves overlap. Notice that the size of the ``hole'' in the middle
of the phase diagrams is almost independent of
the step size. This verifies the result in Theorem~\ref{thm:energy_dissipation_rate_general_integrator},
that the dissipation rate
is asymptotically correct independent of the step size.
The corresponding phase diagrams for
Heun's method (sometimes also called Runge's method)%
\footnote{
The Butcher tableau is
	\begin{tabular}{c|cc}
		0 & & \\
		1 & 1 & \\
		\hline
		& $1/2$ & $1/2$
	\end{tabular}
}
are shown in Figure~\ref{fig:damped_pedulum_phase_Heun}. 
Notice that the dissipation rate in all the simulations
is to small, and depends heavily on the
step size. This reflects the fact that Heun's method
(as most explicit Runge-Kutta methods) assembles ``numerical''
energy.

A detailed plot of the energy error as a function of time, is
shown in the upper graph of Figure~\ref{fig:damped_pendulum_energy_error}. 
(The ``correct'' energy is computed by highly accurate numerical integration.)
In particular, notice
that the energy error for the suggested methods seems to
stay bounded, whereas it grows linearly for Heun's method.
Although this result is not fully explained by our backward error analysis,
it is related to the result obtained in Theorem~\ref{thm:energy_dissipation_rate_general_integrator}
and Theorem~\ref{thm:monotonicity_mod_Ham_alg3S},
that the dissipation rate for the modified Hamiltonian is asymptotically correct
in~$\eps$.

The lower graph in Figure~\ref{fig:damped_pendulum_energy_error} contains comparison of the energy
between the suggested methods, i.e.,
\[
	\abs{H\circ\PhiTwoS^h-H\circ\PhiRKS^h},\;
	\abs{H\circ\PhiThreeS^h-H\circ\PhiTwoS^h},\;
	\abs{H\circ\PhiThreeS^h-H\circ\PhiRKS^h}
\]
(this difference is too detailed
to recognize in the upper graph). Comparing with the values
in the upper graph, notice that the difference
in computed energy between the methods is much smaller
than the actual energy error, i.e., they produce practically
the same result. Also, notice that among the methods,
the difference in result between $\PhiTwoS^h$ and $\PhiRKS^h$ is indistinguishable
as compared to $\PhiThreeS^h$, i.e., $\PhiThreeS^h$ is the ``out-lier''
of the three.

\newcommand{\action}{\phi}

\section{Symmetry and Conservation of Momentum} 
\label{sec:conservation_of_momentum}

In this section we show that the methods suggested in
Section~\ref{sec:numerical_integration_schemes} preserve
invariance under a symmetry, and
conserve corresponding momentum maps.
For preliminaries on geometric mechanics, see~\cite{MaRa1999} 
or~\cite{AbMa1978}.

Let $\G$ be a Lie group acting on $\Q$ by $\action:\G\times\Q\to\Q$.
We write $\action_g := \action(g,\cdot)$.
Further, for $g\in\G$ we use the notation $g\cdot q = \action_g(q)$ and
$g\cdot z = \action_{*g}(z) := (\T^*_{g\cdot q}\action_{g^{-1}})(z)$ for the
lifted action.
The action is associated with a corresponding 
\emph{momentum map}~$\set{J}:\coTQ\to\cog$, where $\cog$ is the dual
of the Lie algebra $\g$ of $\G$.
Explicitly, the momentum map is defined by
\begin{equation}
	\pair{\set{J}(q,p),\xi} = \pair{p,\xi_\Q(q)}, \qquad \forall\; \xi\in\g ,
\end{equation}
where $\xi_\Q$ is the infinitesimal generator of the action on~$\Q$, i.e.,
$\xi_\Q(q) := \frac{\ud}{\ud t}|_{t=0}\big(\exp(t \xi)\cdot q\big)$.

Now, let $Y\in\Xcal(\coTQ)$ and assume that
its flow is \emph{$\G$--equivariant}, i.e., 
that it commutes with the action:
\begin{equation}\label{eq:G-invariance_of_flow}
	\exp(t Y)\circ\action_{*g} = \action_{*g}\circ\exp(t Y).
\end{equation}
The infinitesimal version of this relation is
$[Y,\xi_{\coTQ}]=0$ for all $\xi\in\g$, 
where $\xi_{\coTQ}(z) := \frac{\ud}{\ud t}|_{t=0}\big(\exp(t \xi)\cdot z\big)$
is the infinitesimal generator of the action on~$\coTQ$.
The vector field~$Y$ is then called \emph{$\G$--invariant}.
Since the action on $\coTQ$ is canonical, the vector field
$\xi_{\coTQ}$ is symplectic. In fact, it corresponds
to the Hamiltonian function $J(\xi) = \pair{\set{J}(\cdot),\xi}$,
i.e., $\xi_{\coTQ} = X_{J(\xi)}$ so it is exact symplectic (see~\cite[Sect.~11.2]{MaRa1999}).

Associated with the momentum map is:

\begin{itemize}
	\item The subgroup of $\G$--equivariant diffeomorphisms
	\[
		\DiffG(\coTQ) = \big\{ \varphi\in\Diff(\coTQ); \varphi\circ\action_{*g} = 
		\action_{*g}\circ\varphi, \; \forall\, g\in\G \big\}.
	\] 
	Since $\xi_{\coTQ}$ generates $\action_{*\exp(t\xi)}$, the corresponding
	subalgebra of vector fields is given by
	\[
		\XcalG(\coTQ) = \big\{ X\in\Xcal(\coTQ); 
			[X,\xi_{\coTQ}]=0,\; \forall\, \xi\in\g \big\}.
	\]
	\item The subgroup of momentum conserving diffeomorphisms
	\[
		\DiffJ(\coTQ) = \big\{ \varphi\in\Diff(\coTQ); 
			J(\xi)\circ\varphi = J(\xi),\; \forall\, \xi\in\g \big\},
	\]
	with corresponding subalgebra of vector fields given by
	\[
		\XcalJ(\coTQ) = \big\{ X\in\Xcal(\coTQ); 
			\LieD{X}J(\xi)=0,\; \forall\, \xi\in\g \big\}.
	\]
\end{itemize}


%
Notice that $\G$--equivariance does
not necessarily imply conservation of the momentum map,
not even in the Hamiltonian case.%
\footnote{For example, take $\Q=\G=\realset$, $H(q,p) = q$
and $J(\xi)(q,p)=p\cdot\xi$. Then $[X_H,\xi_{\coTQ}]=[X_H,X_{J(\xi)}]=
X_{\pois{H,J(\xi)}}=0$, but $\LieD{X_H}J(\xi)=\pois{H,J(\xi)}=\xi\neq 0$.}
Nor does conservation of
momentum imply $\G$--equivariance.
However, if a function $H\in\Fcal(\coTQ)$ is
$\G$--invariant, i.e., $\pois{H,J(\xi)}=0$ for all~$\xi\in\g$, 
then for the associated Hamiltonian vector field it holds that 
$X_H\in\XcalG(\coTQ)\cap\XcalJ(\coTQ)$ since
$[X_H,\xi_{\coTQ}]= X_{\pois{H,J(\xi)}}=0$ and $\LieD{X_H}J(\xi) = \pois{H,J(\xi)}=0$.

%

Let $\XcalGJ(\coTQ) := \XcalG(\coTQ)\cap\XcalJ(\coTQ)$.
We now give criterions for the vector fields involved
in system~\eqref{eq:perturbed_Hamiltonian_system_condense}
to be $\G$--invariant and/or momentum preserving.

\begin{proposition}\label{pro:G_invariance_of_triple}
	Let $\mathsf{g}$, $\mathsf{d}$ and $V$ be the tensors
	used in equation~\eqref{eq:perturbed_Hamiltonian_system},
	and $X_T$, $Y$ and $X_V$
	the corresponding vector fields 
	in equation~\eqref{eq:perturbed_Hamiltonian_system_condense}.
	%
	%
	Then,
	\[
		\begin{split}
			\LieD{\xi_\Q}\mathsf{g} = 0
			\quad\quad & \Longrightarrow
			\quad
			X_T \in \XcalGJ(\coTQ)
			\\
			\LieD{\xi_\Q}V = 0
			\quad\quad & \Longrightarrow
			\quad
			X_V \in \XcalGJ(\coTQ)
			\\
			\interior{\xi_\Q}\mathsf{d} = 0
			\quad\quad & \Longrightarrow
			\quad
			Y \in \XcalJ(\coTQ)
			\\
			\left.
			\begin{matrix}
				\LieD{\xi_\Q}\mathsf{g}\!\!\!\! &= 0 \\
				\LieD{\xi_\Q}\mathsf{d}\!\!\!\! &= 0
			\end{matrix}\;
			\right\}
			\quad & \Longrightarrow
			\quad
			Y \in \XcalG(\coTQ)
		\end{split}
	\]
	
\end{proposition}

\begin{proof}
	The result is already well known in the Hamiltonian case of $X_T$ and $X_V$.
	Indeed, $\LieD{\xi_\Q}\mathsf{g}=0$ implies that
	$\G$ acts on $\coTQ$ by isometries.
	In turn, this implies that $T$ is $\G$--invariant.
	Further, $\LieD{\xi_\Q}V = 0$ implies that
	$V$ is $\G$--invariant. In turn,
	$X_T$ and $X_V$ are $\G$--invariant
	and $\pois{T,J(\xi)}=\pois{V,J(\xi)}=0$.
	(See~\cite{Sm1970a} or~\cite[Sect.~4.5]{AbMa1978} or~\cite[Chap.~3]{Ma1992} for details).
	
	The result remains to be shown for the non-Hamiltonian 
	vector field~$Y$. First the third implication.
	Since $Y(q,p)=(0,-\mathsf{d}_q(p^\sharp,\cdot))$ and
	$J(\xi)(q,p) = \pair{p,\xi_\Q(q)}$, it holds
	that $\LieD{Y}J(\xi)(q,p) = -\mathsf{d}_q(p^\sharp,\xi_\Q(q))$.
	Now, $\mathsf{d}_q(p^\sharp,\xi_\Q(q)) = 
	(\interior{\xi_\Q}\mathsf{d})_q(p^\sharp) = 0$ by the presumption.
	Thus, $\LieD{Y}J(\xi)=0$, so $Y\in\XcalJ(\coTQ)$.
	
	Next, the last implication.
	Let $\mathsf{f}$ denote the
	bi-linear form $\coTQ\times\coTQ\to\realset$
	given by $\mathsf{f}_q(p,r):= \mathsf{g}_q(p^\sharp,r^\sharp)=\pair{p,r^\sharp}$.
	If $\LieD{\xi_\Q}\mathsf{g}=0$ then $\LieD{\xi_\Q}\mathsf{f}=0$ (see~\cite[Sect.~3.1]{Ma1992}).
	A calculation equivalent to the one in 
	equation~\eqref{eq:pushforward_calculation} yields
	\[
		(\action_{*g})_*Y(q,p) = 
		\Big(0,- (\action_{*g}\mathsf{d})_q\big((\action_{*g}\mathsf{f})_q(p,\cdot),\cdot\big)\Big).
	\]
	Chose now the path $g(t) = t\xi$, replace $g$ for $g(t)$, and take the time derivative 
	$\frac{\ud}{\ud t}|_{t=0}$. The left hand side then becomes
	the Lie derivative of $Y$ in the direction $\xi_{\coTQ}$.
	Using the product rule on the right hand side we get
	\[
		\LieD{\xi_{\coTQ}}Y(q,p) = 
		\Big(0,-(\LieD{\xi_\Q}\mathsf{d})_q\big(\mathsf{f}_q(p,\cdot),\cdot\big)
		- \mathsf{d}_q\big((\LieD{\xi_\Q}\mathsf{f})_q(p,\cdot),\cdot\big)\Big) .
	\]
	By the presumption the right hand side vanishes, so $[\xi_{\coTQ},Y]=0$. Thus,
	$Y\in\XcalG(\coTQ)$.
\end{proof}

We now have the following basic result, which concerns
conservation of momentum and $\G$--invariance for
the suggested methods.

\begin{proposition}\label{pro:conservation_of_momentum_for_methods}
	Let $\Gsub$ denote any of $\DiffG(\coTQ)$, $\DiffJ(\coTQ)$ or 
	$\DiffGJ{\coTQ}$, and let $\gsub$ denote the corresponding
	subalgebra of vector fields. Then $X_T,Y,X_V\in\gsub$
	implies that
	$\PhiThreeS^h$, $\PhiTwoS^h$ and
	$\PhiRKS^h$, defined by 
	Algorithms~\ref{alg:3S}--\ref{alg:RKS}, 
	belong to $\Gsub$.
\end{proposition}

\begin{proof}
	The result is clear for Algorithm~\ref{alg:3S}--\ref{alg:2S}
	since $\Gsub$ is closed under composition.
	Moreover, for $\PhiRKS^h=
	\exp(h X_T/2)\circ\Psi^h_{X_V + \eps Y}\circ\exp(h X_T/2)$, it holds that
	the $\exp(h X_T/2)\in\Gsub$ if $X_T\in\gsub$.
	The vector field $Y$ defines a linear system on $\T^*_q\Q$.
	Since the momentum map is linear in~$p$, and since all Runge-Kutta
	methods conserve linear invariants, it holds that $\Psi^h_{X_V + \eps Y} \in \DiffJ(\coTQ)$
	if $X_V,Y\in\XcalJ(\coTQ)$. Furthermore, 
	we notice that the lifted action~$\action_{*g}$
	is exactly the co-tangent lift of~$\action_{g^{-1}}$.
	Thus, by Proposition~\ref{pro:coordinate_invariance}
	$\PhiRKS^h$ is invariant under
	the change of coordinates~$(q,p) = \action_{*g}(s,r)$.
	That is, \[ 
		\exp(h (\action_{*g})_*X_T/2)\circ\Psi^h_{(\action_{*g})_*(X_V + 
		\eps Y)}\circ\exp(h (\action_{*g})_*X_T/2)\circ\action_{*g} = \action_{*g}\circ\PhiRKS^h .
	\]
	Now, if $X_T,Y,X_V\in\XcalG(\coTQ)$, then $(\action_{*g})_*X_T = X_T$,
	$(\action_{*g})_*Y = Y$ and $(\action_{*g})_*X_V = X_V$, so
	$\PhiRKS^h\circ \action_{*g} = \action_{*g}\circ\PhiRKS^h$,
	i.e., $\PhiRKS^h\in\DiffG(\coTQ)$.
	%
\end{proof}

In particular, if $\mathsf{g}$, $\mathsf{d}$ and $V$
fulfill the requirements of Proposition~\ref{pro:G_invariance_of_triple},
then it follows by Proposition~\ref{pro:conservation_of_momentum_for_methods} 
that Algorithms~\ref{alg:3S}--\ref{alg:RKS} are
both $\G$--equivariant and momentum conserving.
It is important to point out that this result does not
hold for general exact symplectic integrators,
not even if the integrator
is momentum preserving when $\eps=0$.
For example the alternative splitting method
\begin{equation}
	\Phi_{\textrm{\footnotesize RKS-B}}^h=\exp(h X_V/2)\circ\Psi_{X_T+\eps Y}^h\circ\exp(h X_V/2)	
\end{equation}
reduces to Störmer-Verlet when $\eps=0$,
but does not conserve momentum.


\section{Near Conservation of Relative Equilibria} 
\label{sec:near_conservation_of_relative_equilibria}

In this section we investigate how the methods suggested in
Section~\ref{sec:numerical_integration_schemes} preserve relative
equilibria. 
For background on geometric reduction theory see
e.g.~\cite[Chap.~3]{Ma1992}, \cite[Chap.~4]{AbMa1978}, or~\cite[Chap.~1--2]{MaMiOrRa2007}.

Let $\G$ be a Lie group acting on $\Q$
as in the previous section. We assume that $\G$ acts \emph{freely} and \emph{properly}
on $\G$ (cf.~\cite[Chap.~3]{Ma1992}), which roughly speaking means that
the action is non-singular.

\begin{definition}\label{def:rel_equilibria}
	Let $X\in\Xcal(\coTQ)$. A solution curve $\gamma(t)$ to
	$\dot z = X(z)$ is called a \emph{relative equilibrium}
	if there exists a $\xi\in\g$ such that
	$\gamma(t) = \exp(t\xi)\cdot \gamma(0)$.
\end{definition}

The study of relative equilibria of Hamiltonian systems on
a co-tangent bundle~$\coTQ$ is closely related to
the theory of \emph{co-tangent bundle reduction}.
%
The theme 
is to ``quotient away'' the part of the dynamics generated
by the symmetry group~$\G$, and thus, for each $\mu\in\cog$,
obtain a reduced phase space given by~$\P_\mu=\mathsf{J}^{-1}(\mu)/G_\mu$,
where $\G_\mu$ is the subgroup of $\G$ that leaves $\mu$
invariant under the co-adjoint action, i.e.,
$\G_\mu = \{ g\in\G ; \Ad_{g^{-1}}^* \mu = \mu\}$.

\begin{remark}\label{rem:abelian_reduction}
	If $\G$ is an Abelian Lie group (all elements commute),
	it holds that $\G_\mu = \G$, so
	$\P_\mu = \mathsf{J}^{-1}(\mu)/\G$, and
	$\dim \P_\mu = 2\dim \Q - 2\dim \G$.
	In this case, the reduced phase space $\P_\mu$ is isomorphic to
	$\T^* (\Q/\G)$, equipped with 
	a non-canonical symplectic 
	structure $\omega_\mu$, depending
	on the momentum map value~$\mu$. It holds that
	$\omega_0$ is the canonical symplectic structure
	on $\T^*(\Q/\G)$. In general, $\omega_\mu = \omega_0 + \beta_\mu$,
	for a 2--form~$\beta_\mu$ usually called \emph{magnetic term}.
	In classical mechanical examples, where $\G$ typically is
	a rotation group, $\beta_\mu$ corresponds
	to centrifugal forces due to the rotation.
	If $\dim\G = \dim\Q$ we have~$n$
	first integrals in involutions, i.e., the system
	is Arnold-Liouville integrable.
	The modern notion of co-tangent reduction in the Abelian case
	was developed by Smale~\cite{Sm1970a}. However, its roots goes back
	to Lagrange, Poisson, Jacobi, and Routh.\footnote{
	A historical account of the theory of co-tangent bundle reduction is presented in
	the introduction of the monograph by Marsden~et.~al.~\cite{MaMiOrRa2007}.}
\end{remark}

%
%
%
%

\newcommand{\incl}{\iota}
\newcommand{\dropped}[1]{\overline{#1}}
Let $\incl_\mu:\set J^{-1}(\mu)\to\coTQ$ be the natural inclusion,
and $\pi_\mu:\set J^{-1}(\mu)\to\P_\mu$ the
projection~$\pi_\mu(z) = [z]$, where 
$[z]\in\P_\mu=\set J^{-1}(\mu)/\G_\mu$ is the equivalence class of~$z\in\set J^{-1}(\mu)$.
Given $X\in\XcalGJ(\coTQ)$ and initial data on $\set J^{-1}(\mu)$, 
the flow $\exp(tX)$ 
restricts to the invariant momentum manifold $\set J^{-1}(\mu)$.
Further, due to the symmetry, it drops
to a flow $\exp(t \dropped X)$ on $\Diff(\P_\mu)$
that fulfills $\pi_\mu\circ\exp(t X) = \exp(t\dropped X)\circ\pi_\mu$,
where $\dropped X$ is the vector field on $\P_\mu$ defined
by $\T\pi_\mu\circ X = \dropped X \circ \pi_\mu$.
Notice that for these constructions to make sense,
it is essential that~$X$ is both $\G$--invariant
and momentum conserving.
If $\gamma(t)$ is a relative equilibrium,
then $\pi_\mu\circ\gamma(t) := [\gamma(t)]=[\gamma(0)]$.
Indeed,~$\gamma(t)$
is a relative equilibrium solution of the dynamical system
$\dot z = X(z)$ if and only if $[\gamma(0)]$ is 
an equilibrium of $\dot{[z]}=\dropped X([z])$, i.e.,~$\dropped X([\gamma(0)])=0$.

Note that 
if $\gamma(t)=(q(t),p(t))$ is a relative equilibrium
for system~\eqref{eq:perturbed_Hamiltonian_system_condense},
then it must hold that $\mathsf{d}_{q(t)}(p(t)^\sharp,p(t)^\sharp)=0$ for all~$t$,
since the motion~$\gamma(t)$ conserves the Hamiltonian.

Now, fix some $\eps>0$ and let $Z=X_T+X_V+\eps Y$ be the vector field for the 
system~\eqref{eq:perturbed_Hamiltonian_system_condense},
and assume that $\mathsf{g},\mathsf{d},V$ fulfill all the requirements
in Proposition~\ref{pro:G_invariance_of_triple}, so that~$Z\in\XcalGJ(\coTQ)$.
Further, assume that system~\eqref{eq:perturbed_Hamiltonian_system_condense}
has an asymptotically stable relative equilibrium $\gamma(t)$, i.e.,
the critical point $[\gamma(0)]$ of $\dropped Z$ is asymptotically stable.
In general, an integrator~$\Phi^h$ for~$Z$ \emph{does not} have a nearby
relative equilibrium, not even for very small step sizes. The reason for this
is that $Z$ is not \emph{structurally stable} with respect to arbitrary perturbations
in $\Xcal(\coTQ)$. That is, for a small 
perturbation~$\epsilon Z^\Delta\in\Xcal(\coTQ)$ the vector field $Z+\epsilon Z^\Delta$
does not, in general, possess a perturbed relative equilibrium.
However, the stationary point $[\gamma(0)]$ of the 
reduced vector field $\dropped Z$ \emph{is} structurally
stable with respect to perturbations in $\Xcal(\P_\mu)$.
That is, if $\epsilon\dropped Z^\Delta\in\Xcal(\P_\mu)$, then
for small enough~$\epsilon$, the perturbed reduced vector
field~$\dropped Z + \epsilon \dropped Z^\Delta$ has a nearby stationary
point. As a consequence, the relative equilibrium $\gamma(t)$ of~$Z$
is structurally stable with respect to perturbations $\epsilon Z^\Delta \in \XcalGJ(\coTQ)$,
since then the vector field $Z + \epsilon Z^\Delta$ 
drops to $\dropped Z + \epsilon \dropped Z^\Delta$
under symmetry reduction. The corresponding argument also works
on the level of flow maps. Indeed, let $\Phi^h\in\DiffGJ(\coTQ)$ be a
$\G$--equivariant and momentum preserving integrator, approximating
the exact flow $\exp(h Z)$. By the quotient map, the integrator
drops to a map $\dropped \Phi^h\in\Diff(\P_\mu)$, which approximates
$\exp(h \dropped Z)\in\Diff(\P_\mu)$, and since the fix point
$[\gamma(0)]$ of $\exp(h \dropped Z)$ is hyperbolic, the perturbed map
$\dropped \Phi^h$ has a nearby fix point~$[\gamma_h(0)]$ for~$h$ small enough.
In summary, we have the following:
\begin{proposition}\label{pro:conservation_of_hyperbolic_relative_equilibria}
	Let $\mathsf{g},\mathsf{d},V$
	fulfill all the requirements in 
	Proposition~\ref{pro:G_invariance_of_triple}, and let
	$\Phi^h$ be a $\G$--equivariant and momentum preserving
	integrator for system~\ref{eq:perturbed_Hamiltonian_system_condense}.
	Assume that $\gamma(t)$ is an asymptotically stable
	relative equilibrium.
	Then $\Phi^h$ preserves a nearby relative equilibrium
	for small enough step sizes.
\end{proposition}

Notice that the above result 
heavily relies on hyperbolicity, which in turn requires
that the dissipation parameter~$\eps$ is strictly positive and
the reduced dissipation tensor $\dropped{\mathsf{d}}$ is non-degenerate,
since an equilibrium of a Hamiltonian system (typically) is elliptic.
Thus, the requirement is that the step
size is small enough in relation to the perturbation $\eps Y$, i.e., 
essentially that~$h^r \ll \eps$.

Based on Theorem~\ref{thm:monotonicity_mod_Ham_alg3S}
we now derive a refined result, 
which states that Algorithm~\ref{alg:3S} preserves a stable nearby relative equilibrium
for exponentially long times, independent of~$\eps$.
Recall that a relative equilibrium $\gamma(t)$ is stable if it corresponds
to some local minimum $[\gamma(0)]$ of the reduced
Hamiltonian~$\dropped H$ (that is, $\ud \dropped H([\gamma(0)])=0$ and
the Hessian of $\dropped H$ at $[\gamma(0)]$ is strictly positive
definite).
%
%
%

\begin{theorem}\label{thm:near_preservation_of_stable_relative_equilibria}
	Let $\mathsf{g},\mathsf{d},V$ be real analytic and 
	fulfill all the requirements in 
	Proposition~\ref{pro:G_invariance_of_triple}.
	Assume that $\gamma(t)$ is a relative
	equilibrium of system~\eqref{eq:perturbed_Hamiltonian_system_condense},
	such that the critical point $[\gamma(0)]$ of
	$\dropped Z$ is stable, and let $\dropped{\set U}$ be any compact neighborhood
	of $[\gamma(0)]$ such that the reduced Hamiltonian $\dropped H$ is strictly convex
	on~$\dropped{\set U}$. Then there exists another neighborhood $\dropped{\set V}\subset\dropped{\set U}$
	of $[\gamma(0)]$ and constants $h_0,\kappa,K>0$, independent of~$\eps$, such that
	the reduced numerical solution $[z_0],[z_1],\ldots,[z_n]$ generated by Algorithm~\ref{alg:3S}
	stays in $\dropped{\set U}$ over exponentially long times $n h\leq K \e^{\kappa/h}$
	whenever $[z_0]\in\dropped{\set V}$, $h\leq h_0$ and $\eps \geq 0$.
\end{theorem}

\begin{proof}
	Since the integrator $\PhiThreeS^h$ defined by Algorithm~\ref{alg:3S} is
	$\G$--equivariant and momentum preserving, it drops to
	a reduced integrator $\dropped\PhiThreeS^h\in\Diff(\P_\mu)$.
	Further, there is a
	reduced modified Hamiltonian given by $\dropped H_{h,N}([z]):=H_{h,n}(z)$ for any $z\in[z]$
	(well defined since $H_{h,N}$ is $\G$--invariant). Since $\dropped H$ is strictly
	convex on $\set U$, we may chose $h_0'>0$ such that $\dropped H_{h,N}$ is also
	strictly convex on $\dropped{\set U}$ whenever $h\leq h_0'$, thus attaining a minimum $c_\text{min}=\dropped H_{h,N}([z_{h,N}^*])$ 
	in the interior of~$\dropped{\set U}$. Next, let $c_\text{max}=\inf_{[z]\in\pd\dropped{\set U}}\dropped H_{h,N}([z])$.
	It holds that $c_\text{max}>c_\text{min}$. Now, let $c_\text{mid}=(c_\text{max}-c_\text{min})/2$.
	It holds that the set $\dropped{\set V} = \{ [z]\in\dropped{\set U}: \dropped H_{h,N}([z])\leq c_\text{mid} \}$
	is a neighborhood of $[z^*_{h,N}]$ contained in the interior of $\dropped{\set U}$.
	Let $[z_0]\in\dropped{\set V}$.
	By Theorem~\ref{thm:monotonicity_mod_Ham_alg3S} there exists $h_0'',\gamma,C,\lambda >0$ independent of~$\eps$ such that
	$H_{h,N}([\PhiThreeS^h(z)])-H_{h,N}([z])\leq h C \e^{-\gamma/h}$ whenever 
	$[z]\in\dropped{\set U}$ and $h\leq \min(h_0',h_0'')$,
	which implies that $H_{h,N}([z_n])-H_{h,N}([z_0])\leq n h C \e^{-\gamma/h}$.
	Since $[z_0]\in\dropped{\set V}$ it holds that $H_{h,N}([z_n])\leq c_\text{max}$
	whenever $n h C \e^{-\gamma/h} \leq c_\text{max}-c_\text{mid}$, which yields the result
	since $H_{h,N}([z_n])\leq c_\text{max}$ implies that $[z_n]\in\dropped{\set U}$.
\end{proof}

The result states that if $\gamma(t)$ is a stable relative equilibrium solution
to system~\ref{eq:perturbed_Hamiltonian_system_condense}, then, for small enough step sizes,
the numerical solution generated by Algorithm~\ref{alg:3S} will ``almost'' (for exponentially long times)
preserve a modified relative equilibrium which is close to the exact one.
In contrast to Proposition~\ref{pro:conservation_of_hyperbolic_relative_equilibria},
where the analysis is based on structural stability of hyperbolic
critical points, the step size restriction in Theorem~\ref{thm:near_preservation_of_stable_relative_equilibria} 
is independent of~$\eps$. In particular,
the result holds in the conservative case when~$\eps=0$. Furthermore, no non-degeneracy assumption 
on the dissipation tensor~$\mathsf{d}$ is made: it may be
degenerate in any direction.

\subsection{Relation to Symplectic Integration of Problems with Attracting Invariant Tori} 
\label{sub:relation_to_symplectic_integration_of_problems_with_attracting_invariant_tori}

As mentioned in the introduction, there exists already a well developed theory
for symplectic integration of Arnold-Liouville integrable
systems, perturbed in such a way that only one invariant torus persist,
and becomes attractive. For a thorough treatment of this theory,
see~\cite[Chap.~XII]{HaLuWa2006}.
In this section we give a short review
of that framework, and compare it to the setting and analysis presented in the current paper.
The two approaches are related, but not overlapping.

Consider a perturbation of an integrable Hamiltonian system,
which in \emph{action-angle variables} (cf.~\cite[Chap.~X]{HaLuWa2006})
can be written
\begin{equation}\label{eq:perturbed_integrable_system}
	\begin{split}
		\dot{\vect{a}} &= \eps \vect r(\vect{a},\vect{\theta})
		\\
		\dot{\vect{\theta}} &= \vect\omega(\vect{a}) + \eps \vect\rho(\vect{a},\vect{\theta})
	\end{split}
	\qquad\quad
	\vect{a}\in\realset^n, \vect{\theta}\in\mathbb{T}^n .
\end{equation}
Further, assume there is a point $\vect a^*$ such that
the frequencies $\vect{\omega}(\vect a^*)$ are \emph{diophantine} with exponent~$\nu$ (cf.~\cite[Chap.~X]{HaLuWa2006}),
and such that the \emph{angular average} $\bar{\vect r}(\vect a^*) := \int_{\mathbb{T}^n}\vect{r}(\vect{a}^*,\vect{\theta})\ud\vect\theta$
is small and its Jacobian $A=\bar{\vect r}'(\vect a^*)$ has strictly negative real part.
Then system~\ref{eq:perturbed_integrable_system} has an invariant torus which attracts
a neighborhood of~$\{\vect{a}^*\}\times\mathbb{T}^n$ with exponential rate proportional to~$\eps$.
Now,~\cite[Theorem~5.2, Chap.~XII]{HaLuWa2006} states that if a symplectic integrator
is applied to system~\eqref{eq:perturbed_integrable_system} expressed in
canonical variables $(\vect q,\vect p)\in\realset^{2n}$, then the numerical solution has
a modified attractive invariant torus as long as $h^r \leq c_0 \abs{\log\eps}^{-\kappa}$,
where $\kappa = \max(\nu+n+1,r)$ and~$c_0$ is a constant independent of~$h,\eps$.

Although Arnold-Liouville integrable systems are related to relative equilibria, in the
sense that the former is a special case of the latter (see~Remark~\ref{rem:abelian_reduction}),
systems of the form~\eqref{eq:perturbed_integrable_system} are only overlapping
with systems of the form~\eqref{eq:perturbed_Hamiltonian_system_condense} in the trivial
case when the attractive manifold consists of only one single point (corresponding
to an asymptotically stable equilibrium). Indeed, the setting in this paper
allows for a degenerate dissipation tensor, but has the requirement that
the perturbation vector field $\eps Y$ fulfills $\LieD{Y}H \leq 0$ for
the Hamiltonian~$H$. In contrast, the setting in~\cite[Chap.~XII]{HaLuWa2006}
does not require that the perturbation is monotonely decreasing
energy, but instead that $\bar{\vect{r}}$ in equation~\eqref{eq:perturbed_integrable_system} 
is non-degenerate. As a consequence, in order for system~\eqref{eq:perturbed_Hamiltonian_system_condense}
to overlap with system~\eqref{eq:perturbed_integrable_system} it is needed
that the perturbed vector field fulfills $\LieD{Y}H < 0$, i.e.,
that the dissipation tensor~$\mathsf{d}$ has full rank.
In the common case of systems that have a locally minimal energy point in phase space, this means that 
the corresponding attractive invariant manifold is must be an
asymptotically stable equilibrium point.

As an example, take the Van der Pol equation studied in~\cite[Chap.~XII]{HaLuWa2006}.
Here~$\coTQ=\realset^2$, with
the harmonic oscillator Hamiltonian~$H(q,p) = p^2/2 + q^2/2$, and the
perturbation is described by the tensor~$\mathsf{d}_q(u,v) = (1-q^2)u v$ (i.e., the perturbation is linear in~$p$ 
just as in this paper).
However, for this tensor it \emph{does not} hold that $\LieD{Y}H\leq 0$, since the tensor switches
signature at $q^2=1$. This is indeed the \emph{very reason} that the Van der Pol system has a non-trivial
attractive invariant manifold: the requirement~$\LieD{Y}H< 0$ would force the invariant manifold
to be the equilibrium point~$(q,p) = (0,0)$.
Likewise, the example considered in Section~\ref{sub:elastic_pendulum} below
does not fit the setting in~\cite[Chap.~XII]{HaLuWa2006}, since the
unperturbed system is not integrable, and since the dissipation tensor
is degenerate.

In conclusion, one can say that the main differences between the
type of systems analyzed in~\cite[Chap.~XII]{HaLuWa2006} versus
the ones in this paper, is the form of the perturbations,
and that the
former analysis requires the unperturbed system
to be near integrable, whereas the latter requires the unperturbed system to be of standard form.

Furthermore, the results of the actual numerical analysis
is of different character. This paper gives results on
asymptotically correct dissipation rate
(Theorem~\ref{thm:energy_dissipation_rate_general_integrator} and 
Theorem~\ref{thm:monotonicity_mod_Ham_alg3S}),
and on near preservation of stable relative equilibrium for exponentially
long times independent of the perturbation parameter~$\eps$ 
(Theorem~\ref{thm:near_preservation_of_stable_relative_equilibria}). 
The analysis in~\cite[Chap.~XII]{HaLuWa2006} gives
results on near preservation of attractive invariant tori for indefinite time, but with
step sizes depending weakly on~$\eps>0$.
Also, as pointed out in~\cite[Chap.~XII]{HaLuWa2006}, any numerical 
integrator for system~\eqref{eq:perturbed_integrable_system} preserve a nearby
attractive manifold as long as $h^r \ll \eps$. In contrast, 
for systems with relative equilibria 
the condition $h^r \ll \eps$ is not enough in order for a general integrator to preserve 
nearby relative equilibria.
Indeed, as derived in 
Proposition~\ref{pro:conservation_of_hyperbolic_relative_equilibria} above, the
integrator should be~$\G$--invariant and momentum preserving to assert that
the numerical solution has nearby relative equilibria when~$h^r \ll \eps$.

%
%


\subsection{Numerical Example: Elastic 3--D Pendulum} 
\label{sub:elastic_pendulum}


In this example we consider an elastic pendulum which is affected by gravity.
The configuration space is $\Q=\realset^3\backslash \{ 0\}$.
There is a small Rayleigh damping term for the spring, but otherwise the pendulum
is not damped. Thus, the dissipation is only active in the direction of the spring.
The $\eps$~parameter takes the rôle of the damping coefficient for the spring.
Using Cartesian coordinates 
the equations of motion are:
\begin{equation} \label{eq:elastic_pendulum}
	\left\{
		\begin{split}
			\dot \vq &= \frac{\vp}{m} \\
			\dot \vp &= -k \Big( 1- \frac{\ell}{\abs{\vq}} \Big)\vq + m\vect{g} -
			\eps \frac{\vq \vq^\trans \vp }{m \abs{\vq}^2}
		\end{split}
	\right.
\end{equation}
where $m$ is the mass, $\vect{g}\in\realset^3$ is the gravity, $k$ is the stiffness
of the spring, and $\ell$ is the length of the ``unstretched'' pendulum.
Notice that the damping term only acts in the opposite direction of $\vq$.
Also, when $\vp$ is perpendicular to $\vq$ the damping vanishes.
In terms of the presented framework we have 
\[
	\mathsf{g}_q(u,v)=\frac{m}{2} \vect u^\trans \vect v,\qquad
	\mathsf{d}_q(u,v) = \frac{\vect u^\trans \vq\vq^\trans\vect v}{\abs{\vq}^2},\qquad
	V(q) = \frac{k}{2}(\ell-\abs{\vq})^2-\vq^\trans\vect g .
\]
Throughout the rest of the example, we work with unit constants:
$m=k=\ell=1$, and $\vect g = (0,0,-1)$.

\begin{remark}
	For $\eps=0$ the elastic pendulum problem is a non-integrable
	Hamiltonian system,
	meaning it has chaotic behavior; see~\cite[Chap.~5]{Ho2008a}.
\end{remark}

The problem has an $S^1$--symmetry by rotation about the $z$--axis.
Indeed, the setting is $\G=S^1$ 
with action on $\Q$ generated by $\xi_\Q(\vq) = \xi \, \vect e_3\times\vq$,
where $\vect e_3 = (0,0,1)$.
It is straightforward to check that $\mathsf{g},\mathsf{d}$ and $V$
fulfill all the conditions in Proposition~\ref{pro:G_invariance_of_triple}.
Thus, the flow evolves on $\DiffGJ(\coTQ)$.
The corresponding conserved momentum map is the 
\emph{azimutal angular momentum} given by $J_3(\vq,\vp) = (\vq\times\vp)_3$.
In addition, it follows from Proposition~\ref{pro:conservation_of_momentum_for_methods}
that Algorithms~\ref{alg:3S}--\ref{alg:RKS} applied to this problem
exactly conserves momentum.

The problem is integrated with Algorithms~\ref{alg:3S}--\ref{alg:RKS}
and Heun's method, with the following initial data:
$\vq_0 = (0,1.55884573,-0.6)$ and $\vp_0=(1.34164079,0,0)$.
The energy behavior for various step sizes and
choice of $\eps$ is shown in 
Figure~\ref{fig:elastic_pendulum_plots_energyplot_withlabels}.
Notice that the qualitative behavior of
Algorithms~\ref{alg:3S}--\ref{alg:RKS}, in terms of
mean energy dissipation rate, is superior to the result
with Heun's method. 
Indeed, contrary to Heun's method, the mean energy
dissipation rate for the suggested methods 
does not depend much on the step size.
Figure~\ref{fig:elastic_pendulum_plots_projection_plot}
shows the trajectory of the pendulum in the $x$--$y$--plane.
Notice that that Heun's method does not ``stabilize''
around a limit cycle, which is the case for Algorithms~\ref{alg:3S}--\ref{alg:RKS}.
Instead, the trajectory drifts outwards in an
exponentially increasing fashion.
This is because Heun's method does not preserve a nearby relative
equilibrium, whereas the analyzed methods do. We now
continue with a discussion of relative equilibrium
for the elastic pendulum.

\begin{figure}
	\centering
	{\bf Energy diagrams for elastic pendulum} \\
	\includegraphics[width=0.98\textwidth]{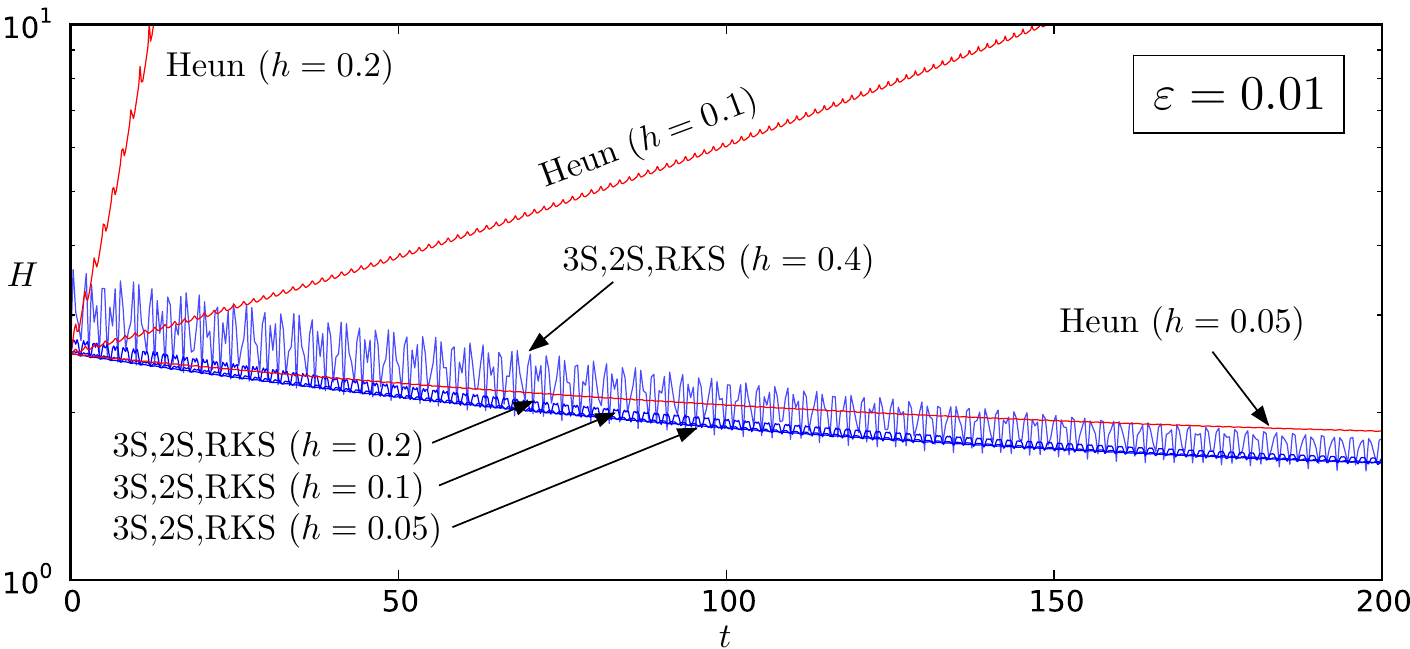}\\
	\includegraphics[width=0.98\textwidth]{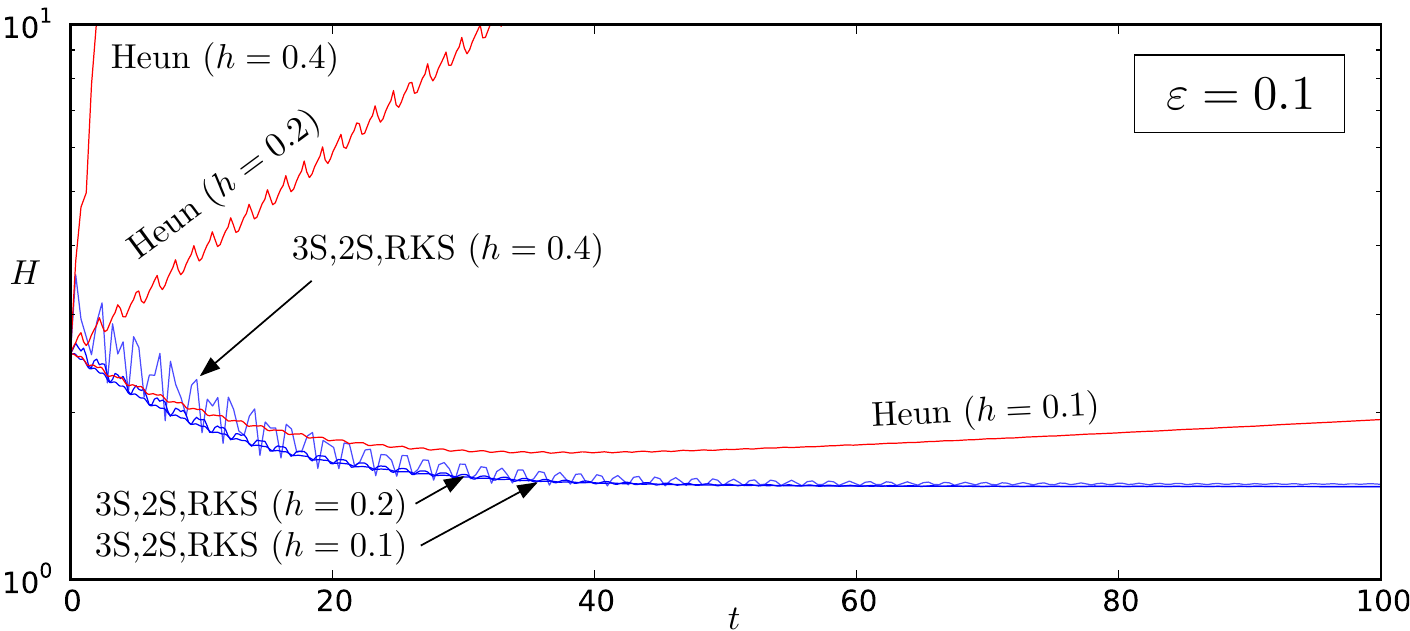}
	\caption{Energy diagrams for the elastic pendulum, using 
	Algorithms~\ref{alg:3S}--\ref{alg:RKS} and Heun's method,
	with $\eps=0.01$ (upper) and $\eps=0.1$ (lower).
	The energy oscillates for Algorithms~\ref{alg:3S}--\ref{alg:RKS},
	but shows that qualitatively correct behavior, independently of the step size.
	For Heun's method, the result depends heavily on the step size,
	and is qualitatively incorrect, since energy is increasing.
	}
	\label{fig:elastic_pendulum_plots_energyplot_withlabels}
\end{figure}

\begin{figure}
	\centering
	{\bf Projected $x$--$y$--trajectory for elastic pendulum} \\[2ex]
	\begin{tabular}{cc}
		\qquad 3S,2S,RKS ($h=0.1$, $\eps=0.1$) & \qquad Heun ($h=0.1$, $\eps=0.1$) \\
		\includegraphics[width=0.45\textwidth]{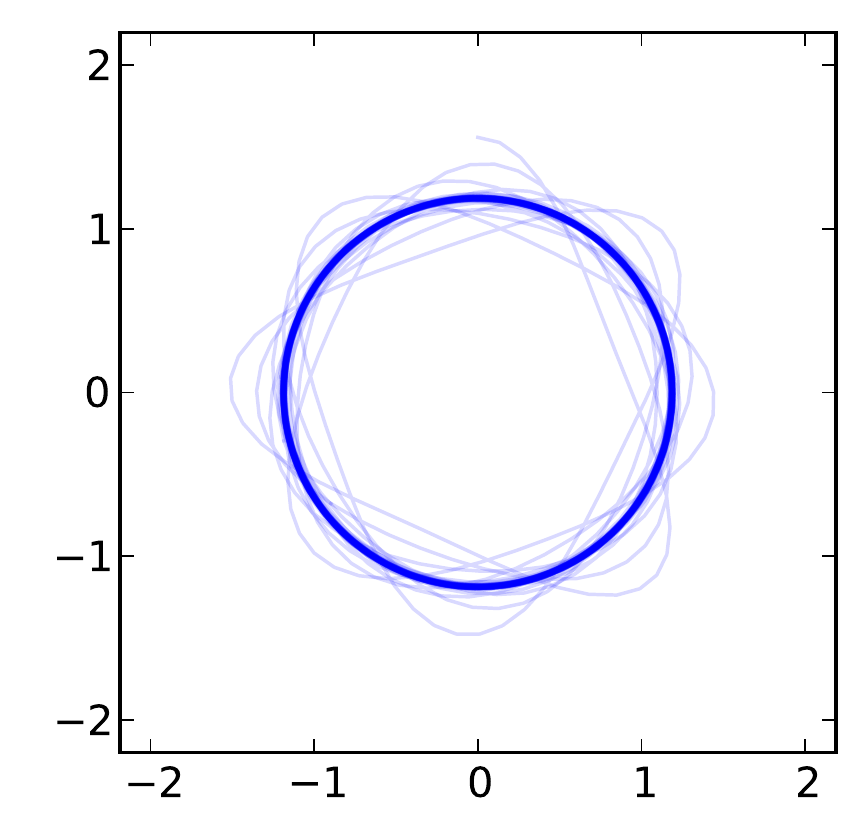} &
		\includegraphics[width=0.45\textwidth]{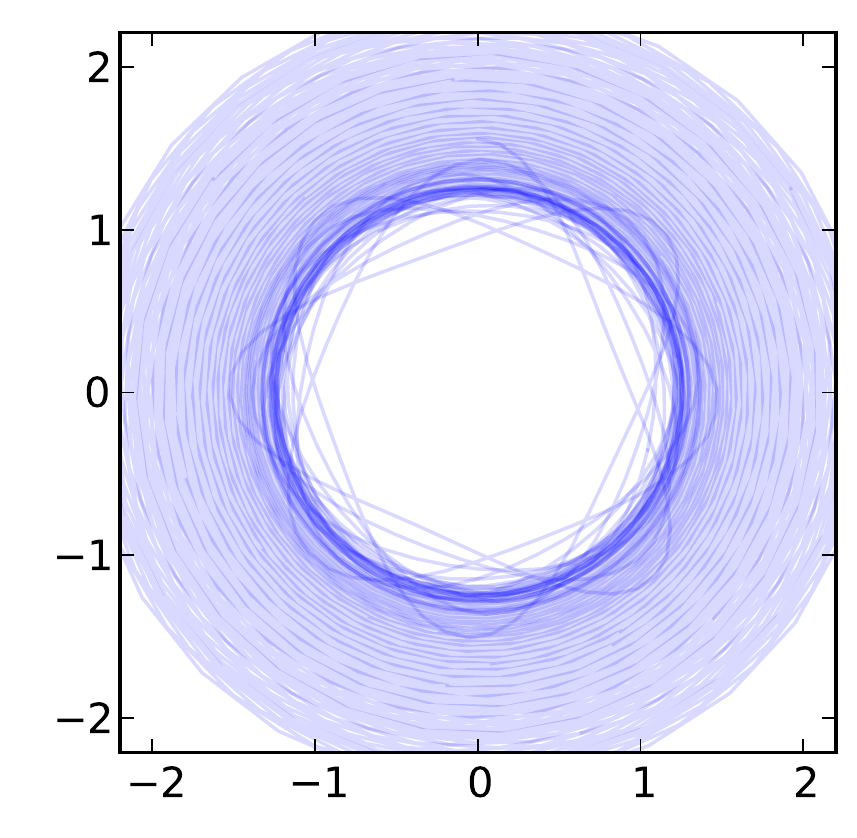} \\
		\qquad 3S,2S,RKS ($h=0.1$, $\eps=0.01$) & \qquad 3S,2S,RKS ($h=0.4$, $\eps=0.01$) \\
\includegraphics[width=0.45\textwidth]{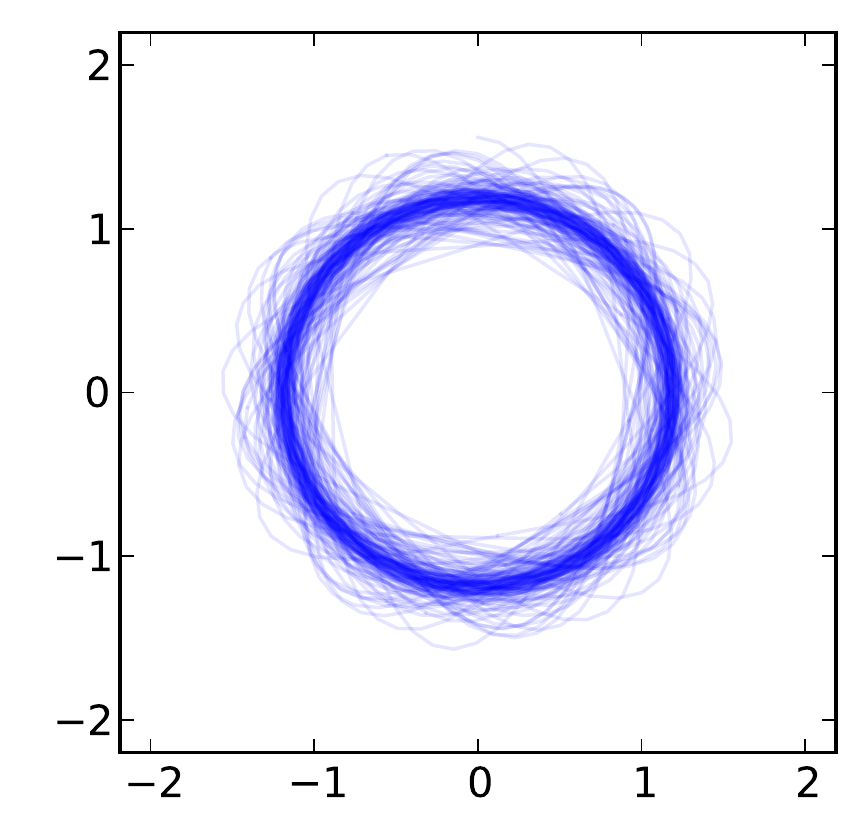} &
\includegraphics[width=0.45\textwidth]{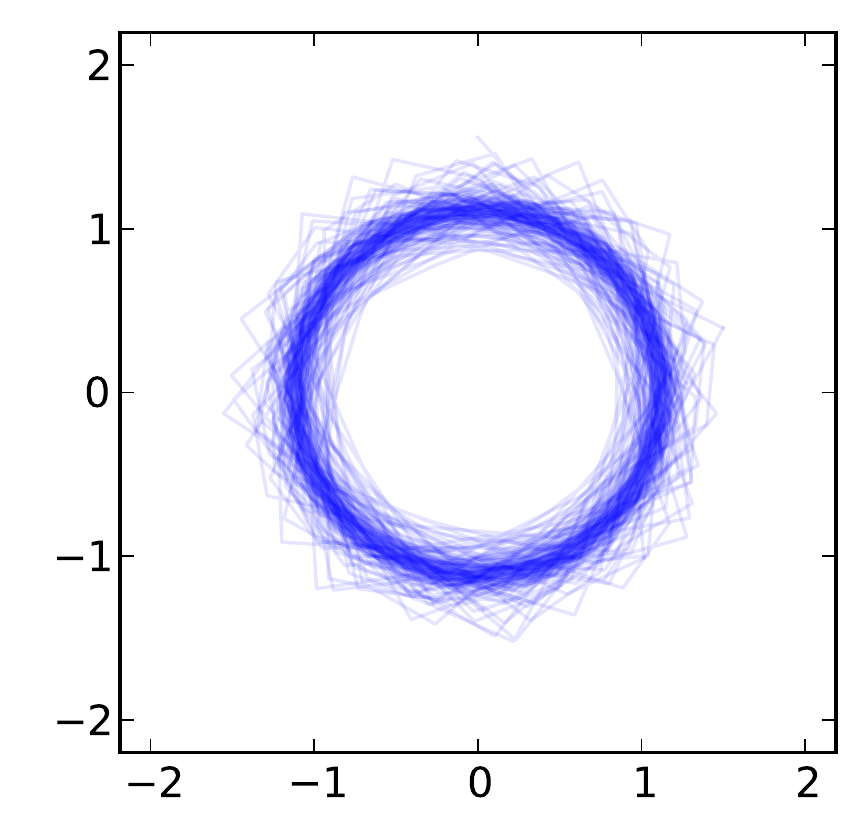}
	\end{tabular}
	\caption{Trajectories in the $x$--$y$--plane for the elastic pendulum,
	using Algorithms~\ref{alg:3S}--\ref{alg:RKS} and Heun's method.
	With Heun's method, the trajectory does not stay bounded,
	but drifts outwards.
	With Algorithms~\ref{alg:3S}--\ref{alg:RKS} the solution approaches
	a limit cycle (corresponding to a relative equilibrium).
	The exact correct limit cycle is not attained, but instead
	a modified one, in accordance with the analysis in
	Section~\ref{sec:near_conservation_of_relative_equilibria}.
	}
	\label{fig:elastic_pendulum_plots_projection_plot}
\end{figure}

\begin{figure} 
	\centering
	{\bf \quad Modified relative equilibria for Algorithms~\ref{alg:3S}--\ref{alg:RKS}} 
	\\[1ex]
	\includegraphics[width=0.98\textwidth]{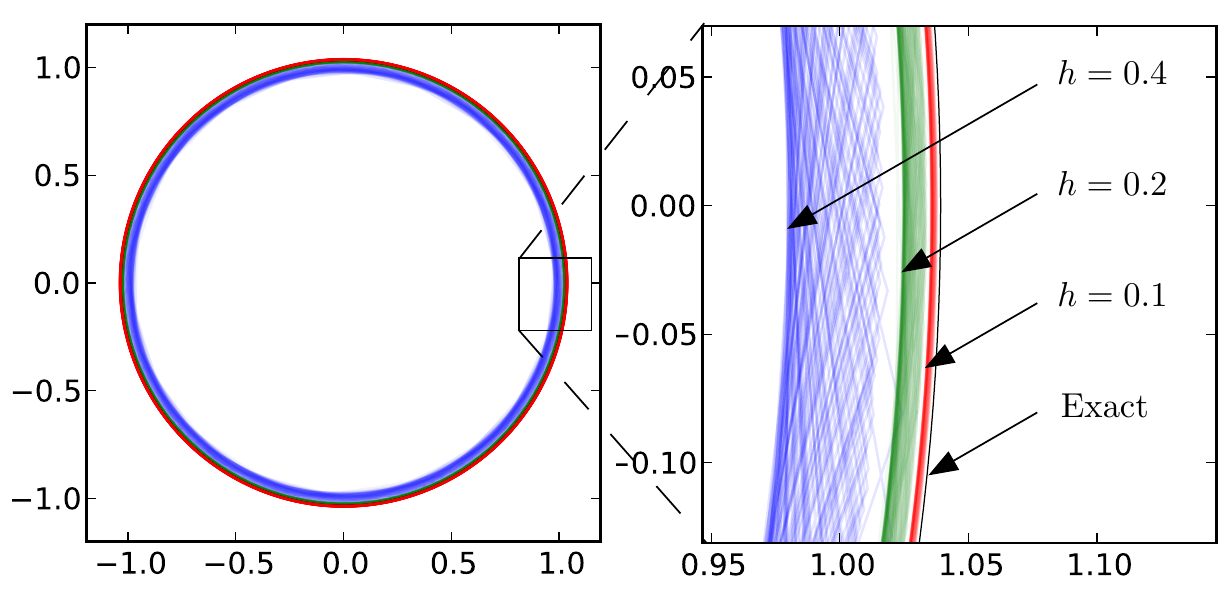}
	\caption{Near preservation of relative equilibria for 
	Algorithms~\ref{alg:3S}--\ref{alg:RKS}, with~$\eps=0.1$ (the results between
	the methods are inseparable). Initial data is given
	on the exact relative equilibrium cycle. The trajectory drifts
	towards a modified relative equilibrium cycle $\mathcal{O}(h^2)$
	away from the exact cycle. The result verifies the
	analysis in Section~\ref{sec:near_conservation_of_relative_equilibria}.
	}
	\label{fig:elastic_pendulum_plots_trajectory_near_rel_eq}
\end{figure}

\paragraph{Relative equilibria}

For each momentum value
$\mu\in\g=\T_\Id S^1 = \realset$ there is a corresponding
relative equilibrium for the elastic pendulum, 
corresponding to rotation with constant
angular velocity about the vertical axis.

Turning to cylindrical coordinates $(\rho,\phi,z)$
on $\Q$, and corresponding co-tangent lifted
momenta $(p_\rho,p_\phi,p_z)$ on $\T^*_q\Q$, the governing
equations~\eqref{eq:elastic_pendulum} take the form
\[
	\left\{
	\begin{split}
		\dot\rho &= \frac{p_\rho}{m}, & \dot p_\rho &= 
			k\rho \Big( \frac{\ell}{\sqrt{\rho^2+z^2}}-1\Big) - \frac{p_\phi^2}{m\rho^3} 
			- \eps\rho \frac{z p_z+\rho  p_{\rho }}{m (z^2+\rho^2)}
		\\
		\dot\phi &= \frac{p_\phi}{m \rho^2}, &  \dot p_\phi &= 0 
		\\
		\dot z &= \frac{p_z}{m}, & \dot p_z &= 
			k z \Big( \frac{\ell}{\sqrt{\rho^2+z^2}}-1\Big) - mg
			- \eps z \frac{z p_z+\rho  p_{\rho }}{m (z^2+\rho^2)}
	\end{split}
	\right.
\]
Notice that $J_3(q,p) = p_\phi$.
Due to the $S^1$--invariance of $\mathsf{g},\mathsf{d},V$,
these equations reduce (in accordance with co-tangent reduction theory)
to a system on $\T(\Q/S^1)$. 
We may use the coordinates $(\rho,z,p_\rho,p_z)$,
in which case the reduced equations are obtained by simply removing
the equations for~$\phi$ and~$p_\phi$, and replacing~$p_\phi$
with the momentum value~$\mu$.

The equilibrium on the reduced space, corresponding to momentum value $\mu$, 
is given for initial data $\gamma(0)=(\rho,z,p_\rho,p_z)$ that fulfill
the relations
%
\[
	\left.
	\begin{split}
		k\rho \Big( \frac{\ell}{\sqrt{\rho^2+z^2}}-1\Big) - \frac{\mu^2}{m\rho^3} & = 0 ,
		& 
		p_\rho & = 0
		\\
		k z \Big( \frac{\ell}{\sqrt{\rho^2+z^2}}-1\Big) - mg & = 0 , 
		&
		p_z &= 0 .
	\end{split}
	\right.
\]
Thus, the corresponding relative equilibrium solution 
$\gamma(t)=\exp(t\xi)\cdot\gamma(0)$
is given~by
\[
	\left\{
	\begin{split}
		\rho(t) &= \rho(0),  & p_\rho(t) &= 0 \\
		\phi(t) & = \phi(0) + \frac{t \mu}{m (\rho(0))^2},  & p_\phi(t) &= \mu \\
		z(t) & = z(0),  & p_z(t) &= 0 .
	\end{split}
	\right.
\]
Since the action is given by~$\exp(t \xi)\cdot (\rho,\phi,z) = (\rho,\phi+t\xi,z)$,
it holds that $\xi=\mu/(m\rho(0)^2)$ for the equilibrium solution.

The $x$--$y$--trajectory using Algorithms~\ref{alg:3S}--\ref{alg:RKS}
and $\eps=0.1$,
with initial data 
\[
	\vq_0 = (0,1.0392304845413263,-0.6),\quad  \vp_0 = (1.3416407864998736,0,0) ,
\]
corresponding
to exact relative equilibrium for momentum
value~$\mu = p_x$,
is given in Figure~\ref{fig:elastic_pendulum_plots_trajectory_near_rel_eq}. 
Again, the result between the methods
is inseparable.
Notice that the trajectories slightly drift off the true
relative equilibrium, toward a modified relative equilibrium.
The rank of~$\mathsf{d}$ is one,
meaning that the dissipation is degenerate even after
the reduction (it is zero if $(z,\rho)$ and $(p_z,p_\rho)$ are orthogonal).
Thus, we cannot apply Proposition~\ref{pro:conservation_of_hyperbolic_relative_equilibria}
to explain the stability of the modified relative equilibrium
for Algorithm~\ref{alg:2S}--\ref{alg:RKS}.
However, notice that the reduced Hamiltonian
is strictly convex in a neighborhood of the
critical point~$[\gamma_h(0)]$.
Thus, by Theorem~\ref{thm:near_preservation_of_stable_relative_equilibria},
the behavior seen in
Figure~\ref{fig:elastic_pendulum_plots_energyplot_withlabels} 
is fully explained for Algorithm~\ref{alg:3S}.

\section{Conclusions} 
\label{sec:conclusions}

For Hamiltonian problems on standard form perturbed by Rayleigh damping,
three numerical integration schemes were analyzed.
When the dissipation parameter is zero, all the methods
reduce to the Störmer-Verlet scheme.
Based on backward error analysis, we gave result on:
\begin{itemize}
	\item Asymptotically correct dissipation rate of the modified
	Hamiltonian (Theorem~\ref{thm:energy_dissipation_rate_general_integrator} and 
	Theorem~\ref{thm:monotonicity_mod_Ham_alg3S}).
	\item Monotone decay of the modified Hamiltonian for small enough
	step sizes (Theorem~\ref{thm:monotonicity_mod_Ham_alg3S}).
	\item Preservation of symmetries, and conservation of momentum
	(Proposition~\ref{pro:conservation_of_momentum_for_methods}).
	\item Near preservation of relative equilibria 
	(Theorem~\ref{thm:near_preservation_of_stable_relative_equilibria}).
\end{itemize}
The theoretical results were verified by numerical examples of
a damped planar pendulum, and a damped elastic 3--D~pendulum.




\section*{Acknowledgment} 
\label{sec:acknowledgment}
The authors would like the thank Claus Führer, Christian Lubich,
and Ignacio Romero for helpful and inspiring discussions during
the work with this paper. 
Also, we would like to thank the reviewers for helpful suggestions.
The work was supported in part by
the Royal Physiographic Society in Lund, Hellmuth Hertz' foundation grant,
and the Marsden Fund in New Zealand.

\bibliographystyle{../../Papers/amsplainshort.bst} 
\bibliography{../../Papers/References}

\def\cprime{$'$}
\begin{thebibliography}{10}

\bibitem{AbMa1978}
R.~Abraham and J.~E. Marsden, \emph{Foundations of Mechanics}, second ed.,
  Benjamin/Cummings Publishing Co. Inc. Advanced Book Program, Reading, Mass.,
  1978.

\bibitem{ArRo2001a}
F.~Armero and I.~Romero, \emph{On the formulation of high-frequency dissipative
  time-stepping algorithms for nonlinear dynamics. {I}. {L}ow-order methods for
  two model problems and nonlinear elastodynamics}, Comput. Methods Appl. Mech.
  Engrg. \textbf{190} (2001), 2603--2649.

\bibitem{ArRo2001b}
F.~Armero and I.~Romero, \emph{On the formulation of high-frequency dissipative
  time-stepping algorithms for nonlinear dynamics. {II}. {S}econd-order
  methods}, Comput. Methods Appl. Mech. Engrg. \textbf{190} (2001), 6783--6824.

\bibitem{Go1996}
O.~Gonzalez, \emph{Time integration and discrete {H}amiltonian systems}, J.
  Nonlinear Sci. \textbf{6} (1996), 449--467.

\bibitem{GoSi1996}
O.~Gonzalez and J.~C. Simo, \emph{On the stability of symplectic and
  energy-momentum algorithms for non-linear {H}amiltonian systems with
  symmetry}, Comput. Methods Appl. Mech. Engrg. \textbf{134} (1996), 197--222.

\bibitem{Gr1988}
J.~Grabowski, \emph{Free subgroups of diffeomorphism groups}, Fund. Math.
  \textbf{131} (1988), 103--121.

\bibitem{HaLuWa2003}
E.~Hairer, C.~Lubich, and G.~Wanner, \emph{Geometric numerical integration
  illustrated by the {S}t\"ormer-{V}erlet method}, Acta Numer. \textbf{12}
  (2003), 399--450.

\bibitem{HaLuWa2006}
E.~Hairer, C.~Lubich, and G.~Wanner, \emph{Geometric Numerical Integration},
  second ed., Springer Series in Computational Mathematics, vol.~31,
  Springer-Verlag, Berlin, 2006.

\bibitem{Ha2011}
A.~C. Hansen, \emph{A theoretical framework for backward error analysis on
  manifolds}, J. Geom. Mech. \textbf{3} (2011), 81--111.

\bibitem{HoOs2010}
M.~Hochbruck and A.~Ostermann, \emph{Exponential integrators}, Acta Numerica
  \textbf{19} (2010), 209--286.

\bibitem{Ho2008a}
D.~D. Holm, \emph{Geometric Mechanics. {P}art {I}}, Imperial College Press,
  London, dynamics and symmetry, 2008.

\bibitem{KaMaOrWe2000}
C.~Kane, J.~E. Marsden, M.~Ortiz, and M.~West, \emph{Variational integrators
  and the {N}ewmark algorithm for conservative and dissipative mechanical
  systems}, Internat. J. Numer. Methods Engrg. \textbf{49} (2000), 1295--1325.

\bibitem{KrMi1997b}
A.~Kriegl and P.~W. Michor, \emph{The Convenient Setting of Global Analysis},
  Mathematical Surveys and Monographs, vol.~53, American Mathematical Society,
  Providence, RI, 1997.

\bibitem{LuWaBr2010}
C.~Lubich, B.~Walther, and B.~Br\"ugmann, \emph{Symplectic integration of
  post-newtonian equations of motion with spin}, Phys. Rev. D \textbf{81}
  (2010), 104025.

\bibitem{Ma1992}
J.~E. Marsden, \emph{Lectures on Mechanics}, London Mathematical Society
  Lecture Note Series, vol. 174, Cambridge University Press, Cambridge, 1992.

\bibitem{MaMiOrRa2007}
J.~E. Marsden, G.~Misio{\l}ek, J.-P. Ortega, M.~Perlmutter, and T.~S. Ratiu,
  \emph{Hamiltonian Reduction by Stages}, Lecture Notes in Mathematics, vol.
  1913, Springer, Berlin, 2007.

\bibitem{MaRa1999}
J.~E. Marsden and T.~S. Ratiu, \emph{Introduction to Mechanics and Symmetry},
  second ed., Texts in Applied Mathematics, vol.~17, Springer-Verlag, New York,
  1999.

\bibitem{McPe2001}
R.~McLachlan and M.~Perlmutter, \emph{Conformal {H}amiltonian systems}, J.
  Geom. Phys. \textbf{39} (2001), 276--300.

\bibitem{McQu2002}
R.~I. McLachlan and G.~R.~W. Quispel, \emph{Splitting methods}, Acta Numer.
  \textbf{11} (2002), 341--434.

\bibitem{McQu2006}
R.~I. McLachlan and G.~R.~W. Quispel, \emph{Geometric integrators for {ODE}s},
  J. Phys. A \textbf{39} (2006), 5251--5285.

\bibitem{Mo2009thesis}
K.~Modin, \emph{Adaptive Geometric Numerical Integration of Mechanical
  Systems}, Ph.D. thesis, Lund University, 2009.

\bibitem{MoPeMaMc2011}
K.~Modin, M.~Perlmutter, S.~Marsland, and R.~I. McLachlan, \emph{On
  {E}uler-{A}rnold equations and totally geodesic subgroups}, J. Geom. Phys.
  \textbf{61} (2011), 1446--1461.

\bibitem{Re1999}
S.~Reich, \emph{Backward error analysis for numerical integrators}, SIAM J.
  Numer. Anal. \textbf{36} (1999), 1549--1570 (electronic).

\bibitem{Sm1970a}
S.~Smale, \emph{Topology and mechanics. {I}}, Invent. Math. \textbf{10} (1970),
  305--331.

\end{thebibliography}


\end{document}